\documentclass [12pt]{amsart}
\usepackage{xcolor,bm}
\usepackage{amssymb,amsmath,amsthm,amsfonts}
\usepackage{enumerate}
\usepackage{setspace}
\usepackage{empheq}
\usepackage[all]{xy}
\usepackage{verbatim}
\usepackage[normalem]{ulem}
\usepackage{todonotes}

\usepackage[bookmarks,colorlinks,breaklinks,backref=page]{hyperref}  
\hypersetup{linkcolor=blue,citecolor=blue,filecolor=blue,urlcolor=blue} 

\numberwithin{equation}{section}
\numberwithin{figure}{section}

\newtheorem{theorem}{Theorem}[section]
\newtheorem{proposition}[theorem]{Proposition}

\newtheorem{lemma}[theorem]{Lemma}

\newtheorem{question}[theorem]{Question} 
\newtheorem{corollary}[theorem]{Corollary}

\theoremstyle{definition}
\newtheorem{definition}[theorem]{Definition}
\newtheorem{example}[theorem]{Example}
\newtheorem{remark}[theorem]{Remark}

\definecolor{myblue}{rgb}{0.6, 0.9, 1}

\makeatletter

\newcommand{\Rmnum}[1]{\expandafter\@slowromancap\romannumeral #1@}
\makeatother

\definecolor{myblue}{rgb}{0.6, 0.9, 1}
\definecolor{mygreen}{rgb}{0,0,1}
\definecolor{purple}{rgb}{0.6,0.2,1}
\definecolor{orange}{rgb}{0.8,0,0.2}

\DeclareMathOperator{\Tr}{Tr}

\newcommand{\bC}{\mathbb{C}}
\newcommand{\bP}{\mathbb{P}}

\newcommand{\bQ}{\mathbb{Q}}
\newcommand{\bR}{\mathbb{R}}

\newcommand{\bN}{\mathbb{N}}

\newcommand{\ord}{\operatorname{ord}}

\newcommand{\eps}{\varepsilon}

\newcommand{\Res}{\operatorname{Res}}

\newcommand{\Kbar}{\overline{K}}
\newcommand{\Qbar}{\overline{\bQ}}

\newcommand{\<}{\langle}
\renewcommand{\>}{\rangle}
\newcommand\Lbar {\overline{L}}

\newcommand\iso{\simeq}

\newcommand{\mbf}{\mathbf}
\newcommand{\bb}{\mathbb}
\newcommand{\fra}{\mathfrak}
\newcommand{\Ph}{\operatorname{Ph}}
\newcommand{\cal}{\mathcal}
\newcommand{\ovl}{\overline}
\newcommand{\Fal}{\mathrm{Fal}}
\newcommand{\h}{\widehat{h}}
\newcommand{\Spec}{\operatorname{Spec}}

\newcommand{\Hom}{\operatorname{Hom}}

\newcommand{\GL}{\operatorname{GL}}
\newcommand{\berkP}{{\mathbb{P}^1_{\operatorname{Berk},v}}}
\newcommand{\berkA}{{\mathbb{A}^1_{\operatorname{Berk},v}}}
\newcommand{\Diag}{\operatorname{Diag}}

\newcommand{\calX}{\mathcal{X}}
\newcommand{\calH}{\mathcal{H}}
\newcommand{\calL}{\mathcal{L}}
\newcommand{\calM}{\mathcal{M}}
\newcommand{\calZ}{\mathcal{Z}}

\theoremstyle{definition} 
\newtheorem{defprop}{Definition/Proposition}[section]

\author{Niki Myrto Mavraki and Jit Wu Yap}

	\email{myrto.mavraki@utoronto.ca}
	\email{jitwuyap@mit.edu}

\begin{document}
\title[]{A quantitative dynamical Zhang fundamental inequality and Bogomolov-type problems}

\begin{abstract}
We prove a quantitative version of Zhang’s fundamental inequality for heights attached to polarizable endomorphisms. As an application, we obtain a gap principle for the Néron--Tate height on abelian varieties over function fields of arbitrary transcendence degree and characteristic zero, extending the result of Gao--Ge--Kühne over $\overline{\mathbb{Q}}$. We also establish instances of effective gap principles for regular polynomial endomorphisms of $\mathbb{P}^2$, in the sense that all constants can are explicit. These yield effective instances of uniformity in the dynamical Bogomolov conjecture in both the arithmetic and geometric settings, including examples in prime characteristic.
\end{abstract}

	\date{\today}

	\maketitle

\section{Introduction}
The dynamical Manin--Mumford and dynamical Bogomolov conjectures, proposed by Shou-Wu Zhang \cite{Zha98}, are far-reaching generalizations of the classical Manin–Mumford and Bogomolov conjectures, proved respectively by Raynaud \cite{Ray83} and by Ullmo \cite{Ulm98} and Zhang \cite{Zha98}. 
In formulating them, Zhang drew upon the analogy between torsion points on abelian varieties and preperiodic points in dynamics.
These conjectures ask whether, given a polarized endomorphism of a projective variety, one can characterize the subvarieties that contain a Zariski-dense set of preperiodic points, or more generally, generic sequences of points of small Call–Silverman canonical height. 
They were later reformulated by Ghioca, Tucker and Zhang \cite{GTZ11,GT} and extended to the setting of function fields of characteristic zero by Gauthier--Vigny \cite{GV24}. In the dynamical setting, both conjectures remain widely open, with only a few cases currently understood; that of polarized endomorphisms of $(\bP^1)^n$ \cite{GNY19,MSW} and of regular polynomial endomorphisms of $\bP^2$ and curves satisfying an assumption regarding their intersection with the line at infinity \cite{DFR23}.

An important tool introduced by Zhang in his proof of the classical Bogomolov conjecture for algebraic tori~\cite{Zha92} is the \emph{fundamental inequality}.  
This was later generalized to adelically metrized line bundles in~\cite{Zha95, Zha95b}, and extended to function fields by Gubler \cite{Gub07b, Gub08}. 
Let $K$ be either a number field or a function field, by which we mean the function field of a regular projective variety defined over an algebraically closed field of arbitrary characteristic. 
Let $\Phi:\bP^N_{\Kbar}\rightarrow\bP^N_{\Kbar}$ be an endomorphism of degree $d \ge 2$, and denote by $\hat{h}_{\Phi}: \bP^N(\Kbar)\to \bR_{\ge 0}$ its associated Call--Silverman canonical height \cite{CS93}, defined by an iteration procedure analogous to the N\'eron-Tate height on abelian varieties. 
Building on the arithmetic intersection theory of Gillet--Soulé~\cite{GS90}, Zhang extended the canonical height $\hat{h}_{\Phi}$ to positive-dimensional subvarieties (see \S\ref{sec: heights} and \S\ref{canonical heights} for precise definitions of heights). 
The fundamental inequality then asserts that for each $\delta < \tfrac{1}{\dim X}$, the set 
\begin{equation}\label{fundamental ineq}
\{\, x \in X(\Kbar) :~ \hat{h}_{\Phi}(x) < \delta \, \hat{h}_{\Phi}(X) \,\}
\end{equation}
is not Zariski dense in $X$. 
This inequality provides a bound to the canonical heights of points on $X$ relative to the canonical height of $X$ itself, at least for points outside a proper closed subset $Z$ of $X$. 
A number of recent works have focused on 
\emph{uniform} or \emph{quantitative} refinements of Bogomolov-type  
phenomena, both in the arithmetic and geometric settings (see for instance 
\cite{BZ, DP99, DP02, DP07, GGK21, Kuh21, DGH21, DKY20, DKY21, LSW21, MS22, GTV23, DM24, Yua24}). 

\smallskip

We prove a quantitative version of Zhang's inequality~\eqref{fundamental ineq}. 
At the expense of reducing $\delta$ we provide an explicit control of the degree of an `exceptional' closed subvariety containing all relatively small points of $X$. 
\begin{theorem}[Quantitative fundamental inequality]\label{quantitative Zhang}
Let $K$ be either a number field or a function field. 
Let $N,d\ge 2$. Let $\Phi:\mathbb{P}_{\overline{K}}^N\to \mathbb{P}_{\overline{K}}^N$ be an endomorphism of degree $d\ge 2$ and  let $X\subset\mathbb{P}_{\overline{K}}^N$ be a subvariety with $r\coloneqq \dim X\ge 1$. Then the set 
\begin{align*}
\left\{x\in X(\overline{K})~:~\hat{h}_{\Phi}(x)<\frac{1}{2(r+1)(4e)^{r+1}}\hat{h}_{\Phi}(X)\right\}
\end{align*}
is contained is a strict subvariety $Z\subsetneq X$ with degree $\deg Z$ at most
\[
3N (r+1) (\deg X)^2 \left((2+6h(N+1))(r+1)^2(4e)^{r+1}\frac{(c_1(N,d) h(\Phi) + c_2(N,d) +c_0(N))}{\hat{h}_{\phi}(X)}+1\right)^{N+1},
\]where $e$ is Euler's number, $h(\Phi)$ denotes the Weil height of the projective point defined by the coefficients of $\Phi$, and we may explicitly take
\begin{align}\label{eq:constants}
\begin{split}
c_0(N) &= \frac{7}{2}h(N+1)=
  \begin{cases}
    \tfrac{7}{2}\log(N+1) & \text{if $K$ is a number field}\\
    0 & \text{if $K$ is a function field},
  \end{cases}\\[6pt]
c_1(N,d) &= (N+1)d^N + 1,\\[4pt]
c_2(N,d) &= 
  \begin{cases}
    (N+1)(d+1)^N (d^N + 1)^{(N+1)(d+2)^N} & \text{if $K$ is a number field}\\
    0 & \text{if $K$ is a function field}.
  \end{cases}
\end{split}
\end{align}
 If $K$ is a function field, then the heights are taken relative to the places of $K$.
\end{theorem}
Recall that $\Phi$-preperiodic points have canonical height zero; in particular our result bounds the geometry of $\Phi$-preperiodic points in $X$. 
In fact we prove an analogous statement for all polarized endomorphisms; see Theorem \ref{quantitative Zhang2}. 
It is clear that the degree of the exceptional $Z$ must depend on $N$, $\dim X$, $\deg X$, and $\deg \Phi$. 
In fact, $(\deg X)^2$ is the optimal dependence on the growth of $\deg X$, as it is easy to construct curves of arbitrarily large height with at least $\frac{1}{2}(\deg X)^2$ many preperiodic points. 
Beyond this inevitable dependence, we expect the remaining factors in our bound to be uniformly bounded in terms of these parameters, except in a few natural counterexamples that will be discussed later; see \S\ref{sec: bigness_question}. 
Hence, Theorem~\ref{quantitative Zhang} can be viewed as essentially optimal. 

As a simple illustration we have the example. 

\begin{example}\label{expicit prep points}
Let $C\subset \bP_{\Qbar}^2$ be a curve given by the vanishing of a degree $2$ irreducible homogeneous polynomial $F_C=0$. 
Let $\Phi: \bP_{\Qbar}^2\to \bP_{\Qbar}^2$ be a degree $2$ endomorphism. 
If the projective height of the point defined by the coefficients of $F_C$ satisfies that 
$$h(F_C)\ge 26 h(\Phi)+ 54\cdot 5^{48}+8 \ge 2(c_1(2,2)h(\Phi)+c_2(2,2)+c_0(2)),$$ 
then 
$$\#(C\cap \mathrm{Prep}(\Phi))\le 48\cdot 4074^3,$$
where $\mathrm{Prep}(\Phi)$ denotes the set of $\Phi$-preperiodic points. 
Here we use our slight strengthening of Ingram's \cite[Theorem 1]{Ing22} in Proposition \ref{height diff} to deduce that $\hat{h}_{\Phi}(C)\ge h(F_C)- c_1(2,2)h(\Phi)-c_2(2,2)-c_0(2)$. 
\end{example}

Our result, albeit the more general version stated as Theorem \ref{quantitative Zhang2}, also allows us to explicitly control the geometry of $\Phi$-preperiodic subvarieties of fixed degree and dimension, as we illustrate in the next example. 
\begin{example}\label{few lines}
Fix a hypersurface in $\bP^2\times\bP^2$ given by the Zariski closure of 
$$H=\{([x:tx:1],[x':tx':1])~:~t,~x,~x'\in\Qbar\}^{\mathrm{Zar}}\subset \mathbb{P}^2\times\mathbb{P}^2.$$
For each $t\in \Qbar$ denote the line with slope $t$ by 
$L_t=\{[x:tx:1]~x\in\Qbar\}^{\mathrm{Zar}}.$
Let $\Phi:\bP_{\Qbar}^2\to\bP_{\Qbar}^2$ be a degree $d\ge 2$ endomorphism. 
We may apply our Theorem \ref{quantitative Zhang2} to find an explicit constant 
$$m\coloneqq m(d,\max\{1,h(\Phi)\}/\hat{h}_{(\Phi,\Phi)}(H)),$$ depending only on $d$ and $\max\{1,h(\Phi)\}/\hat{h}_{(\Phi,\Phi)}(H)$, such that 
\begin{align}\label{prep lines}
\#\{t~:~L_t \text{ is }\Phi\text{-preperiodic}\}\le m.
\end{align}
Indeed, we first apply Theorem \ref{quantitative Zhang2} along with Proposition \ref{height diff} to the split map $(\Phi,\Phi):(\bP^2)^2\to(\bP^2)^2$ and to the hypersurface $H$, to find the explicit constant $m$ such that
\begin{align*}
    H\cap \mathrm{Prep(\Phi,\Phi)}\subset Z\subsetneq H,
\end{align*}
for closed strict subvariety $Z\subsetneq H$ with
$\deg Z\le m$. 
Then we note that for each $t$ such that $L_t$ is $\Phi$-preperiodic we have that $L_t\times L_t\subset H$ has to be one of the $2$-dimensional components of $Z$, leading to \eqref{prep lines}.
\end{example}

We point out that some special cases of our theorem have been previously established. 
Indeed, our proof is inspired by the work of David--Philippon~\cite{DP99}, who proved an analogous result in the setting of tori over~$\Qbar$ (compare with Theorem \ref{David--Philippon--mult}). 
They later used it to prove uniformity in the Bogomolov conjecture for powers of elliptic curves \cite{DP07}.
In the setting of the diagonal curve and for specific endomorphisms of $(\mathbb{P}^1)^2$ corresponding to elliptic curves and to quadratic polynomials, 
such a result has been obtained by DeMarco--Krieger--Ye in~\cite[Theorem 1.8]{DKY20} and~\cite[Theorem 1.9]{DKY21}. 
These results played a key role in the authors' uniform control on the number of common preperiodic points of such maps.   
More recently, Gauthier extended their result to arbitrary endomorphisms of $(\bP^1)^2$, while still focusing on the diagonal curve \cite[Theorem B]{Gau24}. 

The methods in \cite{DKY20,DKY21, Gau24} rely on potential theory as well as explicit H\"older estimates in both archimedean and non-archimedean spaces -- the authors ibid. build on ideas from \cite{FRL06} which lead to quantitative arithmetic equidistribution results.  
Zhang’s original proof of the fundamental inequality employs the arithmetic Hilbert–Samuel theorem \cite{GS92}, the proof of which requires careful use of complex and non-archimedean analysis on the subvariety $V$. Thus, it is difficult to make the required estimates explicit as $V$ varies. A similar difficulty arises when attempting to use the second author’s quantitative equidistribution theorem \cite{Yap24}, which extends the results of \cite{FRL06} to higher dimensional polarized dynamical systems. 

In contrast, our proof follows the more algebraic approach of~\cite{DP99}, 
based on geometry of numbers and elimination theory. 
We interpret dynamical heights of subvarieties as limits of heights of their Chow forms (as in \cite{DP99, Hut09}), as opposed to  arithmetic intersection numbers. 
Using a specialization argument based on the construction of a transcendental generic curve, we reduce the proof to the case where the ground field is either a number field or a function field of transcendence degree one (see \S\ref{sec: generic} and \cite{Gub08}). 
In this setting, we appeal to the absolute Siegel theorem \cite{rth96}, estimates for arithmetic Hilbert functions arising from elimination theory and linear algebra, and to explicit bounds for the difference between canonical and Weil heights as in \cite{Ing22}.

Next, we give two applications of our result towards uniformity in Bogomolov-type problems, both in the classical setting of abelian varieties and in the dynamical context. 

\subsection{Application 1: A geometric gap principle \`a la Gao--Ge--K\"uhne}

As a first application, we establish a geometric gap principle for the N\'eron-Tate height on abelian varieties over characteristic zero function fields. 
Our result is inspired by the work of Gao--Ge--K\"uhne \cite{GGK21}, who established the analogous gap principle over number fields -- leading to uniformity in the Mordell--Lang conjecture. 

Recall that a subvariety $X$ of an abelian variety $A$ is said to \emph{generate} $A$ if $A$ is the smallest abelian subvariety containing $X - X$. 
We write $X^{\circ}$ for the complement in $X$ of all positive-dimensional cosets contained in $X$. This will be a Zariski open subset of $X$, possibly empty. 
 Finally, we let $h_{\Fal}(A)$ denote the absolute (stable) Faltings height of $A$; see \S\ref{heights on av} for our definition. We prove the following.

\begin{theorem}\label{uniform_geometric_Bogomolov}
Let $g\ge 2$ and $D\ge 1$.  
There exist constants $c=c(g,D)>0$ and $m=m(g,D)>0$ with the following property.  
Let $K$ be a function field of characteristic zero, let $A$ be an abelian variety of dimension $g$ defined over $\Kbar$, and let $L$ be an ample symmetric line bundle on $A$.  
Let $X\subseteq A_{\ovl K}$ be an irreducible proper subvariety with $\deg_L(X)\le D$ that generates~$A$, and assume that $X^\circ\neq\emptyset$.  
Then the set
$$
\{\,x\in X(\ovl K)\;:\;\h_{A,L}(x)\le c\, h_{\mathrm{Fal}}(A)\,\}
$$
is contained in a proper subvariety $Y\subsetneq X$ with $\deg_L(Y)\le m$. 
Moreover, $c$ and $m$ can be computed explicitly in terms of $D,g$ and the constants $c_{\mathrm{DGH}}$ from Theorem \ref{HeightGap} and $c_S$ from Theorem \ref{HeightGeometricCompare1}.
\end{theorem}

This result was previously not known outside of the case of curves in their Jacobian, where it was recently established by Yuan \cite{Yua24} and independently by Looper--Silverman--Wilms \cite{LSW21} for one dimensional function fields of arbitrary characteristic. 
The authors in both articles relied on  Cinkir's result \cite{Cin11}, whose proof, as also noted by Gao–Ge–Kühne \cite{GGK21}, seems very hard to generalize to higher dimensions; it is based on explicit formulas for admissible pairings on metrized graphs which are intrinsically one-dimensional. Note that we are also able to handle the case of higher dimensional function fields.

For the proof we appeal to Dimitrov--Gao--Habegger's height inequality \cite[Theorem 1.6]{DGH21}, as presented in Gao--Ge--K\"uhne's \cite[Proposition 4.1]{GGK21}, as well as Zhang's theorem of Successive Minima \cite{Zha98}, and to ideas of \cite{GGK21} to reduce to the case of principally polarized abelian varieties with level structure. 
We argue with induction on the transcendence degree of the function field, and along the way use several 
 specialization results as in \cite{CS93, Ing22, GV24}, as well as a construction of a transcendental curve in \cite{Gub08}. 

Recall also, as pointed out in \cite{GGK21}, that the assumption that $X$ generates $A$ in Theorem~\ref{uniform_geometric_Bogomolov} cannot be omitted, nor can the conclusion be strengthened to a bound on the cardinality of the set of small points.  
However, by restricting to points with height sufficiently close to zero, we obtain the following result, which may be viewed as a uniform geometric Bogomolov conjecture for function fields of characteristic zero for traceless abelian varieties.

\begin{theorem}\label{uniform CGHX}
  Let $g\ge 2$ and $D\ge 1$. Then there exists $N = N(g,D)$ and $\eps = \eps(g,D)$ with the following property. 
  Let $K$ be the function field of a regular projective variety $B$ over an algebraically closed field $k$ of characteristic zero. Fix a very ample line bundle on $B$ defining a Weil height over $K$.  
  Let $A$ be an abelian variety over $K$ with a symmetric ample line bundle $L$ and such that $\Tr_{\ovl{K}/\bb{C}}(A) = 0$. Then for any subvariety $V \subseteq A_{\ovl{K}}$ with $\deg_L V \leq D$, we have
$\#\{x \in V^{\circ}(\ovl{K}) \mid \h_{A,L}(x) \leq \eps\} \le N.$
\end{theorem}

It is worth mentioning that the geometric Bogomolov conjecture was only fully resolved recently, after a series of many works, as Ullmo, and Zhang's approach in the arithmetic setting relied crucially on the distribution of the torsion points in the complex topology. 
For similar reasons, it is very challenging to replicate Gao--Ge--K\"uhne's arguments for uniformity results in the geometric setting, as it builds on Ullmo and Zhang's argument.
The geometric Bogomolov conjecture is established by Cinkir~\cite{Cin11} for curves in their Jacobians, building on earlier work of Zhang~\cite{Zha93, Zha10}. 
Previous special cases were obtained by Moriwaki~\cite{Mor96, Mor97, Mor98}, Yamaki~\cite{Yam02, Yam08}, and Faber~\cite{Fab09}.
For higher-dimensional subvarieties in abelian varieties, the conjecture was proved for abelian varieties with totally degenerate reduction by Gubler~\cite{Gub07b}.
More recently, it was established in all abelian varieties over one-dimensional function fields of characteristic zero by Gao--Habegger~\cite{GH19}, over arbitrary characteristic-zero function fields by Cantat--Gao--Habegger--Xie~\cite{CGHX21}, and finally in all characteristics by Xie--Yuan~\cite{XY22}, using a reduction result of Yamaki~\cite{Yam18}. 
Theorem \ref{uniform CGHX} can then be viewed as a uniform version of the main result in \cite{CGHX21}.

\subsection{Application 2: Explicit gap principles for polynomial endomorphisms of $\bP^2$. }
Our second application lies in the purely dynamical setting and illustrates the ability of our theorem to yield quantitative Bogomolov-type bounds even in positive characteristic, where no prior dynamical uniform result appears in the literature. 

A \emph{regular polynomial endomorphism} $\Phi:\bP^2 \to \bP^2$ is a morphism whose
restriction to the affine chart $\mathbb{A}^2 \coloneqq  \bP^2 \setminus L_{\infty}$ is a polynomial map,
and such that the line at infinity $L_{\infty}$ is totally invariant under $\Phi$. 
We identify the line at infinity with $\bP^1$ and write $\Phi|_{L_{\infty}}:\bP^1\to\bP^1$. 
Regular polynomial endomorphisms have been studied extensively in dynamics; see e.g. \cite{FS94,BJ00,DS10} and the references therein.  
Recently Dujardin--Favre--Ruggiero \cite{DFR23} proved the dynamical Manin--Mumford conjecture for such regular polynomial endomorphisms and for curves that contain a point in $L_{\infty}$ that is not eventually superattracting. 
Here we adopt a different perspective and obtain explicit versions of a Bogomolov-type problem for regular polynomial endomorphisms and Fermat curves. 

\begin{theorem}\label{oberlyff}
Let $d\ge 2$, and for each $n\in\mathbb{N}$ let $C_n: x^n + y^n = z^n$ be the Fermat curve in $\mathbb{P}^2$. 
Let $K$ be a number field or a function field, and let $\Phi:\mathbb{P}^2_{\overline K}\to \mathbb{P}^2_{\overline K}$
be a regular polynomial endomorphism of degree $d\ge 2$ such that $\Phi|_{L_{\infty}}$ has degree at least $2$. 
Then the following hold.
\begin{enumerate}
    \item If $K$ is a function field, $n\ge 2^{21}d^6$, $\mathrm{char}(K)\nmid n$, and $\Phi|_{L_{\infty}}$ is not defined over the field of constants of $K$, then 
    $$\#\left\{x\in C_n(\Qbar):\hat{h}_{\Phi}(x)< \frac{1}{32(4e)^4(\sqrt{d}+1)^2}  h(\Phi|_{L_{\infty}})\right\}\le  2^{79}\, n^2 d^9 \cdot \left(\frac{h(\Phi)}{h(\Phi|_{L_{\infty}})}\right)^3.$$
    \item 
    If $K$ is a number field, $n\ge 2^{1600d^3}$, and $\Phi|_{L_{\infty}}$ is a monic polynomial that is not a power map, then
    $$\#\left\{x\in C_n(\Qbar):\hat{h}_{\Phi}(x)< \frac{1}{32 (4e)^{4}2^{500d^3}}\max\{h(\Phi|_{L_{\infty}}),1\}\right\}\le  2^{1810 d^3}\, n^2 \,
\left(\frac{\max\{1,h(\Phi)\}}{\max\{1,h(\Phi|_{L_{\infty}})\}}\right)^3.$$
\end{enumerate}
\end{theorem}

The appearance of the ratio $h(\Phi)/h(\Phi|{L_{\infty}})$ in the bound is an artifact of the argument. It is expected that the bound should hold without any dependence on this quantity. 
That said, our result already handles several examples. 
\begin{example}
Let $d\ge 2$ and let $p> 2^{21}d^6$ be a prime number. Let
    $$\Phi([X:Y:Z])
  = \bigl[\,P(X,Y)+ZR_1(X,Y,Z) : Q(X,Y) + Z R_2(X,Y,Z) : Z^{d}\,\bigr],$$ 
  for $P, Q\in \overline{\mathbb{F}_p(t)}[X,Y]\setminus \overline{\mathbb{F}_p}[X,Y]$, coprime homogeneous of degree $d\ge 2$, and $R_1,R_2\in\overline{\mathbb{F}_p}[X,Y,Z]$ homogeneous of degree $d-1$.  
Then $$\# (C_{2^{21}d^6}\cap \mathrm{Prep}(\Phi))\le 2^{121}d^{21}.$$
\end{example}

In similar spirit, we have the following example.  
\begin{example}
Let $d\ge 2$ and let consider a degree $d$ cyclotomic purtrubation of a monic polynomial of the form
    $$\Phi([X:Y:Z])
  = \bigl[\,X^d+YP(X,Y)+ZR_1(X,Y,Z) : Y^{d} + Z R_2(X,Y,Z) : Z^{d}\,\bigr],$$ 
  for $P\not\equiv0$ homogeneous polynomial of degree $d-1$ in $\Qbar[X,Y]$, and $R_1,R_2\in \Qbar[X,Y,Z]$ homogeneous of degree $d-1$ with coefficients roots of unity.
Then $$\# (C_{2^{1600d^3}}\cap \mathrm{Prep}(\Phi))\le 2^{5200d^3}.$$
\end{example}


To obtain Theorem \ref{oberlyff}, we apply Theorem \ref{quantitative Zhang} twice for different dynamical systems. 
Along the way we exploit tools from both Arakelov theory and potential theory and invoke ideas from work of Ang--Yap \cite{AY23}, Favre--Rivera--Letelier \cite{FRL06}, and Oberly \cite{Obe24}.
\subsection{Further questions}\label{sec: bigness_question}

We conclude the introduction with a question inspired by the recent results on uniform Bogomolov-type problems \cite{DP02, GGK21, Kuh21, DGH21, DKY20, DKY21, MS22, GTV23, DM24}. 
To state it,  let $V \subset \bP^N$ be a subvariety, let $(Y,g:Y\to Y)$ be a polarized endomorphism, and let 
$p : V \dashrightarrow Y$ be a rational map. 
We say that $(V,Y,g,p)$ is a \textbf{factor of} $\Phi : \bP^N \to \bP^N$ if there exists $n \in \bN$ such that $\Phi^n(V)=V$ and 
$p \circ \Phi^n = g \circ p$ on $V$.
Note that every endomorphism $\Phi$ of $\bP^N$ admits the \textbf{trivial factors} 
$(\bP^N,\bP^N,\Phi^k,\mathrm{id})$ for each $k\in\bN$.  
We say that $\Phi$ is \textbf{irreducible} if it has no non-trivial factors and note that the generic endomorphism of $\bP^N$ will be irreducible as can easily be deduced from the results in \cite{Fak14}. 
Finally, if $X$ is a subvariety of $\mathbb{P}^N$, we say that the pair $(\Phi,X)$ over a 
function field $K$ is \emph{isotrivial} if, after a change of coordinates, both $\Phi$ and $X$ 
are defined over the constant field of $K$. For endomorphisms defined over number fields this 
is an empty notion.

\begin{question}\label{bigness_question}
Let $d\ge 2$, $D\ge 1$, and let $K$ be either a number field or a function field.  
Does there exist $\epsilon=\epsilon(d,D)>0$ such that
\[
\widehat{h}_{\Phi}(X)\;\ge\;
\epsilon\bigl(c_1(N,d)\,h(\Phi)+c_2(N,d)+c_0(N)\bigr)
\]
for every irreducible endomorphism $\Phi:\mathbb{P}^N_{\overline{K}}\to\mathbb{P}^N_{\overline{K}}$ 
of degree $d$ and every subvariety $X\subset\mathbb{P}^N_{\overline{K}}$ with $\deg X\le D$ 
such that $(\Phi,X)$ is not isotrivial?
\end{question}

A positive answer to this question, combined with Theorem~\ref{quantitative Zhang}, would 
yield full uniformity in the dynamical Bogomolov conjecture for irreducible endomorphisms.

Although we expect more general forms of Question~\ref{bigness_question} to admit affirmative 
answers, we restrict ourselves to irreducible endomorphisms, as even this special case appears 
highly nontrivial.  
The reader may compare Question~\ref{bigness_question} with the results 
of Gao--Ge--Kühne \cite{GGK21} and their consequences (see Remark~\ref{HeightGapRemark} and 
Theorem~\ref{GeometricHeightGap2}), noting in particular the necessary assumption in 
\cite{GGK21} that the subvariety generates the ambient abelian variety.

Finally, we remark that the expert reader may notice relations between this question and bigness notions in Yuan--Zhang’s theory of 
metrized line bundles on quasi-projective varieties \cite{YZ24} -- this is not accidental.  
Under 
appropriate bigness assumptions, Yuan--Zhang obtain strong forms of uniformity, and several 
recent works, including \cite{MS22,DM24,GTV23,Yua24,GGK21, Kuh21}, may be viewed as steps toward 
establishing such bigness statements.  The conceptual distinction in our setting is that the 
principle sought in Question~\ref{bigness_question} is \emph{pointwise}: we do not require 
families of dynamical systems, yet we still aim to obtain uniform control over the geometry 
of small points.

\subsection{Structure of the article} 
We begin in \S\ref{sec: heights} with a review of height theory over number fields and function fields, together with two complementary constructions of heights for positive-dimensional subvarieties. In \S\ref{sec: ffofcurves} we extend the quantitative fundamental inequality of David–Philippon \cite{DP99} for the Weil height to the setting of function fields of curves. Canonical heights are introduced in \S\ref{sec: DPhtrick}, where we also establish preliminary comparison results. Theorem \ref{quantitative Zhang} is proved in \S\ref{sec: DPhtrick}, using a generic specialization argument of Gubler to pass to arbitrary function fields, together with results of Ingram \cite{Ing22}. The dynamical applications are developed in \S\ref{sec: DynamicalBogomolov}, where we prove Theorem \ref{oberlyff}. Finally, in \S\ref{sec: Bogomolov} we combine the height inequalities of Dimitrov–Gao–Habegger \cite{DGH21} and Gao–Ge–K\"uhne \cite{GGK21} with specialization techniques \cite{CS93, Ing22, GV24} to deduce Theorems \ref{uniform_geometric_Bogomolov} and \ref{uniform CGHX} from Theorem \ref{quantitative Zhang}.
\subsection{Acknowledgements}
We are indebted to Xinyi Yuan for alerting us to errors in the literature, for sharing his counterexample with Guo in \cite{GY25} and for his feedback on a previous version. We also thank Laura DeMarco and Ruoyi Guo for helpful discussions. N.M. Mavraki would like to acknowledge the support of NSERC.
\section{Heights}\label{sec: heights}
In this section we recall the basic definitions of heights and establish preliminary results. 
\subsection{Heights of points}
Let $n\in\bN$. Let $n\in \mathbb{N}$.
For $[x_0:\ldots:x_n]\in\mathbb{P}^{n}(\overline{\mathbb{Q}})$ we define the absolute logarithmic Weil height
$$h([x_0:\ldots:x_n])=h(x_0,\ldots,x_n)=\sum_{v \in M_L} \frac{[L_v:\bQ_v]}{[L:\bQ]} \log \max\{ |x_0|_v,|x_1|_v,\ldots,|x_n|_v\},$$
where $(x_0,\ldots,x_n)\in L^n$ for some finite extension $L$ of $\bQ$.

We will also work with heights in function fields. 
Let $k$ be an algebraically closed field of arbitrary characteristic and set $K=k(B)$, the function field of a normal projective variety $B$ over $k$.
Fix an ample line bundle $\mathcal{M}$ on $B$. 
Following \cite[1.4.6]{BG06} we endow $K$ with a natural set of absolute values $M_K$ with respect to $\cal{M}$. 
Each place $v\in M_K$ corresponds to a prime divisor $Z$ of $B$. 
For any prime divisor $Z$ of $B$, the local ring $O_{B,Z}$ is a discrete valuation ring and so the order of vanishing of $f \in K^{\times}$ on $Z$, denoted by $\ord_Z(f)$, is well defined. 
We can then define an absolute value 
$$|f|_Z = e^{- \ord_Z(f) \deg_{\cal{M}}(Z)}$$
where $\deg_{\cal{M}}(Z)=\cal{M}|_{Z}^{\cdot \dim Z}$.  As $Z$ varies over all prime divisors of $B$, the collection of absolute values $M_K$ satisfies the product formula, that is for all $f\in K\setminus \{0\}$, we have 
\begin{align*}
\sum_Z \ord_Z(f)\deg_{\cal{M}}(Z)=0. 
\end{align*}
We define the height of $[f_0:\cdots:f_n]\in\bP^n(K)$, with $f_i\in K$ for all $i$, by 
$$h_K([f_0:\cdots:f_n])=h_K(f_0,\ldots,f_n)=\sum_{v\in M_K}\log \max \{|f_0|_v,\ldots,|f_n|_v\}.$$

By \cite[Remark 1.4.11]{BG06}, for a finite extension $L/K$ the normalization of $B$ in $L$ yields a normal projective variety $B'$ with a finite surjective morphism $\varphi:B'\to B$ and function field $k(B')=L$ and the absolute values $|\cdot|_w$ in $M_L$ constructed as above with respect to $\varphi^{*}\cal M$ agree with the absolute values extending $M_K$ as defined in \cite[\S 1.3.6]{BG06}. 
Furthermore, if $n_w = [L_w:K_v]$ for $w|v$, then the normalized absolute values $\{| \cdot |^{n_w}_w\}_{w \in M_L}$ satisfy the product formula \cite[Prop. 1.4.2]{BG06}; see also \cite[Ex. 1.4.13]{BG06}.
Moreover, for each $P\in \bP^n(K)$ we have 
$h_L(f)=[L:K]h_K(f)$.
We may thus define $h_K:\mathbb{P}^n(\Kbar)\to\mathbb{R}_{\ge 0}$ by 
\begin{align}\label{weilheightff}
h_K(P)=\frac{h_L(P)}{[L:K]},~\text{for } P\in\bP^n(L).
\end{align}

\textbf{From now on $K_0$ denotes either $\mathbb{Q}$ or a fixed base function field $k(B)$ with an ample line bundle $\mathcal{M}$ on $B$.
Unless stated otherwise, all heights are computed with respect to the places of $M_{K_0}$ and we write $h$ for $h_{K_0}$. }

\subsection{Heights of varieties}
We next define heights of higher–dimensional subvarieties of $\mathbb{P}^n$.
Two comparable notions will be used in this article—one due to Philippon \cite{Phi95} (extended to function fields by Gubler \cite{Gub97, Gub98}), and another arising from Arakelov intersection theory and developed by Bost–Gillet–Soulé \cite{BGS94}, and by Zhang \cite{Zha95, Zha95b}.

Firstly, we define the height of a homogeneous ideal $I\subset \overline{K}_0[x_0,\ldots,x_n]$ as well as the height of subvariety $X\subseteq \mathbb{P}^n$ defined over $\overline{K}_0$, following Philippon \cite{Phi95} over number fields and Gubler \cite{Gub97, Gub98} over function fields. 
Let $A \coloneqq  \overline{K}_0[x_0,\ldots,x_n]$ and suppose that  the homogeneous ideal $I\subset A$ has rank $n+1-r$, equivalently $\dim Z(I)=n-r$. For each $i\in\{1,\ldots,r\}$, we let $U_i = u_{i,0}x_0 + \cdots + u_{i,n}x_n$ be a generic hyperplane form in indeterminates $u_{i,0},\ldots,u_{i,n}$. We then consider the elimination ideal
$$\mathfrak{C}_{\bf{1}}I = \displaystyle\bigcup_{\delta \geq 0} (I[\mbf{1}] :_{\overline{K}_0[u_{i,j}]} A_{\delta}),$$
where $I[\mathbf{1}]$ is the ideal generated by $I$ along with the linear forms $U_1,\ldots,U_{r}$ inside $A[u_{i,j}]$ and for each $\delta\ge 0$, we denote by $A_{\delta}$ the $\overline{K}_0$-module consisting of the elements of $A$ that are homogeneous of degree $\delta$. It turns out that $\mathfrak{C}_{\mbf{1}}I$ is a principal ideal inside $\overline{K}_0[u_{i,j}]$, generated by a homogeneous polynomial $f_{I[\mbf{1}]}$ and we define 
$$h(I)\coloneqq h(f_{I[\mbf{1}]}).$$ 
For a homogeneous polynomial $P\in \overline{K}_0[x_0,\ldots,x_n]$, the value $h(P)$ denotes the projective height of the coefficient vector of $P$.  
If $X$ is a subvariety of $\mathbb{P}_{\overline{K}_0}^n$ given as the zero locus of a homogeneous ideal $I_X$, then 
$f_{I_X[\mbf{1}]}$ is precisely the Chow form of $X$ (see  \cite{DS95}), most often denoted by $\mathrm{Ch}_X$, and we define 
$$h(X)\coloneqq h(I_X)=h(\mathrm{Ch}_X).$$ 
Note that in the case of $K_0=\bQ$, our definition differs than that given by Philippon in \cite{Phi95} and used in \cite[\S 2.1]{DP99}; at the archimedean places they consider the Euclidian norm instead of the max norm. 
\emph{We write $h_{\mathrm{Ph}}(X)$ to refer to the Philippon height of $X$ in the case of number fields, with the convention that for function fields $h_{\mathrm{Ph}}(X)=h(X)$.}

Later on we will need a higher degree version of Philippon's height of a homogeneous ideal $I\subset \overline{K}_0[x_0,\ldots,x_n]$ of rank $n+1-r$ (c.f. \cite[\S 4.1]{DP99}), which we introduce now, though the reader can skip this definition for the moment. Fix a tuple $\mbf{d} = (d_1,\ldots, d_{r}) \in\bb{N}^{r}$ and for each $i\in\{1,\ldots,r\}$ consider generic hypersurfaces $U_{i,d_i}$ of degree $d_i$, where $U_{i,d_i}$ has an indeterminate `coefficient' $u_{i,j}$ for each degree $d_i$ monomial $\underline{x}^{\mbf{\alpha}}$ where $|\alpha| = d_i$. Denote by $I[\mbf{d}]$ the ideal in $A[u_{i,j}]$ generated by the elements of $I$ as well as the forms $U_{i,d_i}$. We may then define the elimination ideal
$$\fra{C}_{\mbf{d}} I \coloneqq  \displaystyle\bigcup_{\delta \geq 0} (I[\mbf{d}]:_{\ovl{K}_0[u_{i,j}]} A_{\delta})$$
which is again a principal ideal of $\ovl{K}_0[u_{i,j}]$, generated by $f_{I[\mbf{d}]}$. We then define the degree $\mbf{d}$-height of $I$, by 
\begin{align}\label{degreedheight}
h_{[\mbf{d}]}(I)\coloneqq h(f_{I[\mbf{d}]}).
\end{align} 
Note that if $X$ is a subvariety of $\bb{P}^n$, given by the zero locus of an ideal $I_X$, then the principal ideal $\fra{C}_{\mbf{d}} I_X$ parametrizes the space of degree $d_i$ hypersurfaces $H_i$ for which $X \cap H_1 \cap \cdots \cap H_{r} \not = \emptyset$. 
Moreover, $\deg X=\deg f_{I_X}\coloneqq d(I_X)$.

We will often work with the normalized heights
\begin{align}\label{normalization}
\hat{h}(X)=\frac{h(X)}{\deg X},~\hat{h}_{\mathrm{Ph}}(X)=\frac{h_{\mathrm{Ph}}(X)}{\deg X},
\end{align}
and we will often slightly abuse the notation to write that $\hat{h}_{\mathrm{Ph}}(X)=\hat{h}(X)$ over a function field. 

We now describe an alternative approach to defining heights of subvarieties when $K_0$ is a function field, via intersection theory on models. This approach was developed by Gubler \cite{Gub98, Gub03}.
In the number–field setting, an analogous definition is available using adelic metrized line bundles and Arakelov intersection theory (Bost–Gillet–Soulé \cite{BGS94}, Zhang \cite{Zha95b}).

Recall that $K_0=k(B)$ is our `base' function field and that we have fixed an ample line bundle $\cal{M}$ on $B$. 
Let $X \subseteq \bb{P}^n_{K_0}$ be a subvariety that is defined over $K_0$ and let $\cal{L}$ be a line bundle on $\bP^n_B\coloneqq \mathbb{P}^n_k\times B$ that extends $O_{\bP^n_{K_0}}(1)$. We define the height of $X$ relative to $\cal{L}$ (and the choice of $\cal{M}$ on $B$) by 
$$h_{\cal{L}}(X) \coloneqq  \cal{L}^{\dim X+1} \cdot \cal{M}^{\dim B - 1} \cdot \cal{X},$$
where $\cal{X}\subset \bP^n_B$ denotes the Zariski closure of $X\subset \bP^n_{K_0}\subset \bP^n_B$. 
Similarly, we define the normalized height of $X$ by 
$$\hat{h}_{\cal{L}}(X)=\frac{h_{\cal{L}}(X)}{\deg_{O(1)}(X)}.$$
It is easy to see that we can extend this definition to subvarieties of $\bP^n_{\overline{K_0}}$ similar to \eqref{weilheightff}. 
Indeed, if a subvariety $X\subset \bP^n_{\overline{K_0}}$ is defined over a finite extension $L$ of $K_0$, and $B'$ is the normalization of $B$ in $L$ with a finite surjective morphism $\varphi: B' \to B$ and $L = k(B')$, then we define the height of $X$ with respect to $\cal{L}$ by 
\begin{align}
    h_{\mathcal{L}}(X)=\frac{1}{[L:K_0]}\cal{L}'^{\dim Y+1} \cdot \phi^{*}(\cal{M})^{\dim B - 1} \cdot \cal{Y}',
\end{align}
where $\cal{X}'\subset\bP^n_{B'}$ is the Zariski closure of $X\subset \bP^n_L\subset \bP^n_k\times B'$ and $\cal{L}'$ is the pullback of $\cal{L}$ via the natural map $\bP^n_{B'}\to\bP^n_B$ induced by $\phi$; see \cite[Theorem 3.5]{Gub08}.

If $\cal{L}$ is the relative $O(1)$ in $\bb{P}^n_B$, denoted by $\cal{O}(1)$, and $X$ is a point, then this definition agrees with the one in \eqref{weilheightff}; see \cite[Example 2.4.11]{BG06}. More generally for a subvariety $X\subset{P}^n_{\overline{K_0}}$ of arbitrary dimension we have 
\begin{align}\label{geometric_height_comparison}
 h_{\mathrm{Ph}}(X)=h_{\mathcal{O}(1)}(X), ~ h_{[(d_1,\ldots, d_{n-\dim X})]}(X)=d_1\cdots d_{n-\dim X}h_{\mathcal{O}(1)}(X); 
\end{align}
see \cite[Example 1.2 and \S 3]{Gub97} and \cite[\S 3]{Gub08}.

More generally, building on Arakelov theory and the fundamental work of Bost--Gille--Soul\'e, Zhang \cite[Theorem (1.4)]{Zha95b}, Gubler \cite{Gub98,Gub07b, Gub08} and Chambert-Loir-Thuillier \cite{CL06, CLT09} developed an intersection theory for integrable adelic metrized line bundles for number fields and function fields respectively. 
	So denoting the set of $n\ge 0$ dimensional cycles of an irreducible projective variety $Y$ over $K_0$ by $Z_n(Y)$, their respective intersection numbers are written as
	\begin{align}\label{intersection}
		\begin{split}
			\widehat{\mathrm{Pic}}(Y)_{\mathrm{int}}^{n+1}\times Z_n(Y)&\to \mathbb{R}\\
			(\overline{L}_1,\cdots,\overline{L}_n,D)&\mapsto \overline{L}_1\cdots\overline{L}_n\cdot D=\overline{L}_1|_D\cdots\overline{L}_n|_D,
		\end{split}
	\end{align}
	where $\widehat{\mathrm{Pic}}(Y)$ denotes the space of integrable adelic metrized line bundles on $Y$ as in \cite{YZ17, YZ24}.  
    If $X$ is an irreducible closed subvariety of $Y$ defined over $\overline{K_0}$, we can define the height of a subvariety associated to any $\overline{L}\in \widehat{\mathrm{Pic}}(Y)_{\mathrm{int}}$, as follows. Denote by $X_{\mathrm{gal}}$ the closed $K_0$-subvariety of $Y$ corresponding to $X$ and write 
	
	\begin{align}\label{height def}
		\hat{h}_{\overline{L}}(X)=\frac{\overline{L}^{\dim Z+1}\cdot X_{\mathrm{gal}}}{\deg_{L}(X_{\mathrm{gal}})}\in \bR.
	\end{align}
   
To recast our heights in this language, we let $\overline{O(1)}$ denote the line bundle $O(1)$ in $\bP_{K_0}^n$ equipped with the model metric induced by the ample line bundle
$\mathcal{M}$ on $B$ as in \cite[\S 2.3]{BG06} or \cite[§3]{Gub08}, in the case of function fields and with the standard adelic metric as in \cite[Ex.~3.2]{Yua08}
in the case of number fields. 
In both cases we refer to $\overline{O(1)}$ as the \textbf{standard metric on $O(1)$.}
For a subvariety $X$ of $\bP^n_{\overline{K_0}}$ we write 
\[
\hat{h}_{\mathrm{st}}(X)
   \coloneqq  \hat{h}_{\overline{O(1)}}(X).
\]
This construction extends the classical Weil height of points and agrees,
in the function--field setting, with the intersection--theoretic height introduced earlier
(see \cite[Thm.~3.5]{Gub08}), so that for function fields 
$\hat{h}_{\mathrm{st}}\equiv\hat{h}\equiv\hat{h}_{\mathrm{Ph}}.$
In the case of number fields, equality may fail; however, the heights remain comparable. 
Indeed by \cite[Proposition 2.1 (v)]{DP99}, and our afformentioned observation, we have that for any  subvariety $X\subset \bP^N$ over either function field or number field, 
\begin{align}\label{nf heights compare}
|\hat{h}_{\Ph}(X)-\hat{h}_{\mathrm{st}}(X)|\le c_0(N)(\dim X+1),
\end{align}
where
$c_0(N)\coloneqq \frac{7}{2}h(N+1),$
which is $0$ for function fields.


\section{Revisiting David--Philippon's inequality for the Weil height: the case of function fields}\label{sec: ffofcurves}
Our aim in this section is to prove the following control on the number of small points with respect to the Weil height. It will be crucial in the proof of Theorem \ref{quantitative Zhang}, as we shall see in \S\ref{sec: DPhtrick}.

\begin{theorem}\label{David--Philippon--mult}
Let $K_0$ be either $\bQ$ or the function field of a curve and let $n\in\bN$. 
Let $X\subset \mathbb{P}^n_{\overline{K_0}}$ be an irreducible subvariety that is $D$-nice (see Definition \ref{Dnice}). Let $\delta \geq 2D+\dim X+2$ and let $\epsilon>0$. Then there exists a homogeneous form $q \in  \overline{K_0}[x_0,\ldots,x_n]$ of degree $\delta$ that doesn't annihilate $X$ and such that 
if $$\hat{h}_{\mathrm{Ph}}(X)\ge (\dim X+1)(4e)^{\dim X+1}\left(\frac{2\epsilon}{\delta}+6h(n+1)\right),$$ then 
\begin{align}
\left\{ x \in X(\overline{K_0}):   h_{\mathrm{Ph}}(x) \leq \frac{1}{(4e)^{\dim X+1}}\hat{h}_{\mathrm{Ph}}(X) \right\}\subset Z(q).
\end{align}
\end{theorem}
In fact we will prove a stronger result stated in \S\ref{construction of small section} as Proposition \ref{small section}.  If $K_0=\bQ$, then this is treated by David--Philippon \cite[Corollary 4.12 ]{DP99}, so it remains to handle the case of function fields. Our strategy follows the one in \cite{DP99} to which we shall often refer.

Henceforth in this section \textbf{$K_0$ is a function field of a curve}.

\subsection{The absolute Siegel theorem and a small linear form}\label{geometry of numbers}
The first key ingredient is the absolute Siegel theorem in \cite{rth96} which serves as a replacement of \cite[Proposition 4.8]{DP99}. 
To state it we recall the definition of the `Schmidt' height of a vector space. Let $n\in\bN$ and let $V$ be a subspace of $(\overline{K}_0)^{n+1}$. If $\dim V = 1$ and is generated by $\vec{x}=(x_0,\ldots,x_n)\in (\overline{K}_0)^{n+1}$, then we define its Schmidt height to be $$h_S(V)\coloneqq h(x_0,\ldots,x_n).$$ By the product formula this is independent of the choice of spanning vector $\vec{x}$. More generally, suppose that $\dim V=r\in\{1,\ldots,n+1\}$. Then its $r$-th wedge power $\bigwedge^r V$ defines an $1$-dimensional subspace of $\bigwedge^r (\overline{K}_0)^{n+1}$. Letting $e_1,\ldots,e_{n+1}$ be the standard basis of $(\overline{K}_0)^{n+1}$, we identify  $\bigwedge^r (\overline{K}_0)^{n+1}$ with $\overline{K}_0^{\binom{n+1}{r}}$ by identifying the basis vectors $e_1\wedge \ldots\wedge e_{i_r}$ with $i_1 < \cdots < i_r$ with the canonical basis of $\overline{K}_0^{\binom{n+1}{r}}$ considering the lexicographic order. The Schmidt height of $V$, denoted by $h_S(V)$, is then defined to be the height of a generator of $\bigwedge^r V$ under this identification. In particular, it is the height of a vector whose coordinates are the Pl\"ucker coordinates of $v_1\wedge \cdots\wedge v_r$ for a basis $\{v_1,\ldots,v_r\}$ of $V$. 
We can now recall Roy--Thunder's absolute Siegel lemma. 
\begin{theorem}\label{DPh:SubspaceSmall}{\cite[Corollary 8.2]{rth96}}
Let $K_0$ be the function field of a curve. Let $n,m\in\bN$ and let $V \subseteq \overline{K}_0^{n+1}$ be a subspace of dimension $m$. Let $\epsilon>0$. Then,  there exists subspaces 
$$W_1\subset W_2\subset\cdots\subset W_m=V,$$  such that for each $i\in\{1,\ldots,m\}$, we have  $\dim W_i=i$ and 
$$h_S(W_i) \leq \frac{i}{m} h_S(V) +\epsilon.$$
\end{theorem}

As a consequence we are able to find a `small' linear form with respect to the Schmidt height of a given subspace $V \subseteq\overline{K}_0^{n+1}$. Here we measure the size of a linear form $q$ defined over a finite extension $L$ of $K_0$ that doesn't vanish identically on $V$, by 
$$h_{\widetilde{V}}(q) = \sum_{w \in M_L} \frac{[L_w:(K_0)_v]}{[L:K_0]} \sup_{x \in V(\Lbar)} \log \frac{|q(x)|_w}{|x|_w},$$
where we have fixed a extension of $|\cdot|_w$ to $\Lbar$ and $v$ denotes the place of $K_0$ below $w\in M_L$
. This quantity is independent of the choice of $L$. The following result will be crucial for our proof.

\begin{corollary} \label{DPh:SmallLinear}
Let $K_0$ be the function field of a curve.
Let $V \subseteq \overline{K}_0^{n+1}$ be a subspace of dimension $m$ and let $\epsilon>0$. Then there exists a linear form $q$ defined over $\overline{K}_0$ that does not vanish identically on $V$ and such that 
$$h_{\widetilde{V}}(q) \leq -\frac{h_S(V)}{m}+\epsilon.$$
\end{corollary}

\begin{proof}
By Proposition \ref{DPh:SubspaceSmall}, we may find a subspace $V$ of dimension $m-1$ such that $$h_S(W) \leq \frac{m-1}{m} h_S(V)+\epsilon.$$ Choose a linear form $q$ such that $W= \{q = 0 \} \cap V$ and suppose that $q$ is defined over a finite extension $L$ of $K_0$. We claim that this $q$ is our desired linear form.

To see this, we first choose a basis $v_1,\ldots,v_{m-1}$ for $W$ and let $x \in V$ be an element such that $q(x) = 1$. Then $x$ is uniquely defined up to addition by an element in $W$, so that $$y\coloneqq v_1 \wedge \cdots \wedge  v_{m-1}\wedge x\in \bigwedge^m V$$ is  independent of the choice of $x\in V(\overline{L})$ with $q(x)=1$. Let $$z\coloneqq v_1 \wedge \cdots \wedge v_{d-1} \in \bigwedge^{m-1} W.$$
Writing $x$ and $v_i$'s in the standard basis, we know that for each $w\in M_L$, the sup norm $\|y\|_w$ is the maximum of the $v$-adic absolute values of them determinants of the $m \times m$ minors of the matrix representing $\{x,v_1,\ldots,v_{m-1}\}$. 
Similarly, $\|z\|_w$ is the maximum $|\cdot |_v$ of the determinants of the $(m-1) \times (m-1)$ minors for the matrix representing $\{v_1,\ldots,v_{d-1}\}$. Using the Laplace expansion (see also \cite[Lemma 3.4]{rth96}), we obtain that 
$$\log \|y\|_w \leq \log \|z\|_w + \log \|x\|_w.$$
Therefore, for any $x\in V$ with $q(x) = 1$, we have 
$$\log\frac{|q(x)|_w}{\|x\|_w} \leq - \log \|y\|_w + \log \|z\|_w$$
Since we may scale $x$ such that $q(x) = 1$ is satisfied, we infer that
$$\sup_{x \in V} \log\frac{|q(x)|_w}{\|x\|_w} \leq - \log \|y\|_w+ \log \|z\|_w.$$
Summing up over all $w$, we get
$$h_{\widetilde{V}}(q) \leq h_S(W) - h_S(V) \leq -\frac{h_S(V)}{d}+\epsilon,$$
as desired.
\end{proof}
\subsection{Estimates for the Hilbert functions}\label{elimination theory}
Recall that $I\subset \overline{K_0}[x_0,\ldots,x_n]=A$ is a homogeneous ideal of rank $n+1-r$. 
The geometric Hilbert function is given by
$$H_g(I,\delta) = \dim_{\overline{K_0}} (A/I)_{\delta}.$$ 
We will need the following result of Chardin. 
\begin{proposition}{\cite[TH\'EOR\`EME]{Cha89}}\label{Chardin}
Let $I\subset \overline{K_0}[x_0,\ldots,x_n]$ be a homogeneous ideal of rank $n+1-r$ and degree $d(I)$. Then 
\begin{align*}
H_g(I,\delta)\le  d(I) \binom{\delta + r-1}{r-1},
\end{align*}
for all $\delta\in\bN$.
\end{proposition}

Next, we recall that the arithmetic Hilbert function 
$$H_a(I,\delta) = h_S(I_{\delta}),$$ where we view $I_{\delta}$ as a vector space in $\overline{K_0}^{\binom{\delta +n}{n}}$, where each coordinate represents the coefficient of a monomial of degree $\delta$in our variables $x_0,\ldots,x_n$. The next lemma is the crucial result needed to construct our small section.
It provides a lower bound for the arithmetic Hilbert function in terms of the height of the ideal and $\delta$ under the assumption that the ideal is $D$-nice. It is precisely the counterpart of \cite[Proposition 4.2]{DP99} in the setting of function fields.

Following David--Philippon \cite[\S 4.1]{DP99}, we recall the definition of a $D$-nice ideal. 
\begin{definition}\label{Dnice}
A homogeneous ideal $I \subseteq \overline{K_0}[x_0,\ldots,x_n$] is $D$-nice for $D\in\bN$ if for all $\mbf{d}=(d_1,\ldots,d_r) \in \bb{N}^{r}$, we have 
$(I[\mbf{d}] :_{\overline{K_0}[u_{i,j}]} A_{D+ d_1 + \cdots + d_r - r + 1}) \not = (0).$
We say that a subvariety $X=Z(I_X)\subset \mathbb{P}^n_{\overline{K_0}}$ is $D$-nice if its defining ideal $I_X$ is $D$-nice.
\end{definition}

\begin{proposition} \label{ArithHilbertBound}
Let $K_0$ be the function field of a curve. 
Let $I$ be a homogeneous ideal with rank $n+1-r \leq n$ which is $D$-nice. Then for all $\delta \geq D+1$, we have 
$$H_a(I;\delta) \geq h(I) \cdot \left(\frac{\delta - D - 1}{r} \right)^r.$$
\end{proposition}

\begin{proof}
Set $\delta_0 = \lfloor \frac{\delta - D - 1}{r} \rfloor + 1$ and let $\mbf{d}_0= (\delta_0,\ldots,\delta_0) \in \bb{N}^{r}$. Since $I$ is $D$-nice, and $\delta\ge D-r\delta_0-r+1$, we know there exists $g \in \fra{C}_{\mbf{d}_0} I\setminus\{0\}\subset  \overline{K_0}[u_{i,j}]\setminus\{0\}$ such that 
\begin{align}\label{inclusion}
g A_{\delta} \subseteq (I, U_{1,\delta_0},\ldots,U_{r,\delta_0})=I[\mbf{d}_0] 
\end{align}
where we take the ideal generated in $A[u_{i,j}]$ generated by $I$ and the $U_{i,\delta_0}$'s. We define $I[\mbf{d}_0]_{\delta}$ to be the elements in $I[\mbf{d}_0]$ that are homogeneous and of degree $\delta$ in the $x_i$'s. Then $gA_{\delta}$ is a subset of $I[\mbf{d}_0]_{\delta}$ and $I[\mbf{d}_0]_{\delta}$ is a $\ovl{K_0}[u_{i,j}]$-module.

Let $L(u_{i,j})$ denote the field of fractions of $\ovl{K_0}[u_{i,j}]$. Then 
$$ I[\mbf{d}_0]_{\delta} \otimes_{\ovl{K_0}[u_{i,j}]} L(u_{i,j}) \subseteq A_{\delta} \otimes_{\ovl{K_0}} L(u_{i,j})$$
and $gA_{\delta} \subseteq I[\mbf{d}_0]_{\delta}$ implies that 
$$A_{\delta} \otimes_{\ovl{K_0}} L(u_{i,j}) = I[\mbf{d}_0]_{\delta} \otimes_{\ovl{K_0}[u_{i,j}]} L(u_{i,j})$$
as $g$ is invertible in $L(u_{i,j})$. We thus have equality. This will allow us to construct an explicit $g' \in \ovl{K}_0[u_{i,j}]$ such that $g' A_{\delta} \subseteq I[\mbf{d}_0]$ as follows.

We first pick a $L(u_{i,j})$-basis of $I[\mbf{d}_0]_{\delta} \otimes_{\ovl{K_0}[u_{i,j}]} L(u_{i,j})$. Consider first a basis of $I_{\delta}$ over $\ovl{K_0}$. We then append to this basis all elements of the form $x_0^{a_0} \cdots x_n^{a_n} U_{i,\delta_0}$ where $\sum a_i = \delta - \delta_0$ to form a set $\cal{B}$. Then clearly $\cal{B}$ spans $I[\mbf{d_0}]$ as a $\ovl{K_0}[u_{i,j}]$ and so spans $I[\mbf{d}_0]_{\delta} \otimes_{\ovl{K_0}[u_{i,j}]} L(u_{i,j})$ as a $L(u_{i,j})$-vector space. We may then choose a $L(u_{i,j})$-basis $\cal{B}'\subset \cal{B}$ out of this spanning set consisting of a basis of $I_{\delta}$ over $\ovl{K_0}$ along with some elements of the form $x_0^{a_0} \cdots x_n^{a_n} U_{i,\delta_0}$.

Now we express $\cal{B}'$ in terms of the monomial basis $x_0^{b_0} \cdots x_n^{b_n}$ of $A_{\delta} \otimes_{\ovl{K_0}} L(u_{i,j})$. This gives us a square matrix $M$ with entries in $\ovl{K_0}[u_{i,j}]$. By Cramer's rule, we know that $\operatorname{adj}(M) \cdot M = \det(M) \operatorname{Id}$ and in particular, if $g' = \det(M)$ which is non-zero, then $g' A_{\delta} \subseteq I[\mbf{d}_0]$ as desired since $\operatorname{adj}(M)$ is a matrix with entries in $\ovl{K_0}[u_{i,j}]$.

In particular, $g'\in \fra{C}_{\mbf{d}}(I)$ and if we let $f$ denote a generator of $ \fra{C}_{\mbf{d}}(I)$, then clearly $f|g'$ and $h(g')\ge h(f)=h_{\mbf{d}_0}(I)$. In view of \eqref{geometric_height_comparison} and the definition of $\delta_0$, we thus get
\begin{align}
h(g') \geq h(f) \geq \delta_0^r h(I) \geq \left( \frac{\delta - D - 1}{r} \right)^r h(I).
\end{align}
We will now show that 
\begin{align}
h(g')\le h_S(I_{\delta})=H_a(I;\delta),
\end{align}
to complete the proof. 
Observe that for any basis element of the form $x_0^{a_0} \cdots x_n^{a_n} U_{i,\delta_0}$, the corresponding entries in our matrix $M$ is either $0$ or some $u_{i,j}$. The first $N\coloneqq \dim_{\overline{K_0}} I_{\delta} $ rows of the matrix $M$ correspond to the basis elements from $I_{\delta}$, and in particular are in $\overline{K_0}$. We may do a Laplace expansion along the first $N$ rows and by the strong triangle inequality, we obtain for each place $v$ (of an appropriate finite extension of $K_0$), the maximum $v$-adic absolute value of the coefficients of $\det(M)$ viewed as a polynomial in the $u_{i,j}'s$ is bounded from above by the maximum of $|\det(M')|_v$ where $M'$ is a $N \times N$ minor of $M$ from the first $N$ rows. But this is exactly $|\wedge^{N} I_{\delta}|_v$ and so $h(g') \leq h_S(I_{\delta})$, as desired.
\end{proof}

\subsection{Construction of the small section}\label{construction of small section}

We now apply the results from \S\ref{geometry of numbers} and \S\ref{elimination theory} to construct our small section. 
Recall that $A=\overline{K_0}[x_0,\ldots,x_n]$. Given a homogeneous ideal $I\subset A$ and an element $q \in A\setminus I$ that is homogeneous of degree $\delta>0$, we define 
$$h_I(q) \coloneqq  \sum_{w \in M_L} \frac{[L_w:K_v]}{[L:K]} \log\left(\sup_{x \in Z(I)} \frac{|q(x)|_v}{||x||^{\delta}_v}\right).$$
We note here that this definition defers than the one in \cite[page 522]{DP99} in the case of a number field, but is more similar tho the one in \cite{Phi95}.

\begin{proposition} \label{SmallSection}
Let $I \subseteq A = \overline{K_0}[x_0,\ldots,x_n]$ be a homogeneous prime ideal of rank $n+1-r \leq n$ that is $D$-nice. Let $\delta \geq D+1$ and let $\epsilon>0$. Then there exists a homogeneous form $q \in A \setminus I$ of degree $\delta>0$ such that 
\begin{align}\label{small section}
h_{I}(q) \leq -\frac{h(I)}{re^r d(I)} \cdot \frac{(\delta - D - 1)^r}{(\delta +r -1)^{r-1}}+\epsilon.
\end{align}
In particular, if $\delta \geq 2D+r+1$ and $\frac{h(I)}{d(I)}\ge r(4e)^r\frac{2\epsilon}{\delta}$, then all the points $x \in Z(I)(\overline{K_0})$ with 
$$h(x) \leq \frac{h(I)}{d(I)} (4e)^{-r}$$
are zeros of $q$.
\end{proposition}
\begin{proof}
Let $\delta\ge D+1$ and let $\epsilon>0$. 
We view $I_{\delta}$ as a $\overline{K_0}$-vector space in $\overline{K_0}^{\binom{n+\delta}{\delta}}$ and let $I^{\perp}_{\delta}$ denote its orthogonal complement with respect to the standard inner product $\<\cdot,\cdot\>$ on $\overline{K_0}^{\binom{n+\delta}{\delta}}$. 
By the duality theorem \cite[Theorem 1]{rth96}, we have that $h_S(I_{\delta})=h_S( I^{\perp}_{\delta})$. Then by Proposition \ref{ArithHilbertBound} we infer that
\begin{align}\label{one}
h_S(I^{\perp}_{\delta}) \geq h(I) \left(\frac{\delta - D-1}{r} \right)^r.
\end{align}
We now apply Corollary \ref{DPh:SmallLinear} to $I^{\perp}_{\delta}$ to find a linear form $q$, which we may identify as an element of $A_{\delta} \setminus I_{\delta}$, such that 
\begin{align}\label{two}
h_{\widetilde{I^{\perp}_{\delta}}}(q) \leq -\frac{h_S(I^{\perp}_{\delta})}{H_g(I;\delta)}+\epsilon.
\end{align}
Here we used the fact that $I^{\perp}_{\delta} $ and $(A/I)_{\delta}$ are isomorphic vector spaces so that $\dim I^{\perp}_{\delta} = H_g(I;\delta)$. We will now show that this $q$ is our desired small section.

Given $x = [x_0: \cdots :x_n]\in Z(I)$, we may consider the vector $\vec{w}\coloneqq  (x_0^{a_0} \cdots x_n^{a_n})_{\mathbf{a}} $ where $\mathbf{a}$ consists of all tuples $(a_0,\ldots,a_n)$ with $\sum a_i = \delta$ and $a_i\ge 0$. Then $\vec{w}$ lies in $I^{\perp}_{\delta}$ and for each place $v$ in a sufficiently large extension of $K_0$, we have 
$$\frac{|q(x)|_v}{||x||^{\delta}_v} = \frac{|\< q , \vec{w} \>|_v}{\|\vec{w}\|_v}.$$
Thus taking logarithms and supremum on both sides and summing up over all places, we obtain 
\begin{align}\label{three}
h_I(q) \leq h_{\widetilde{I^{\perp}_{\delta}}}(q).
\end{align}
The desired inequality in \eqref{small section} now follows combining \eqref{one}, \eqref{two}, \eqref{three} and Theorem \ref{Chardin}.

To prove the second part of the proposition, notice that if $x \in Z(I)$ is not a zero of our form $q$, then we may use $q$ to evaluate the height of $x$ and obtain 
$$-\delta h(x) = \sum_{w \in L} \frac{[L_w:K_v]}{[L:K]} \log \frac{|q(x)|_v}{||x||^{\delta}_v} \leq h_I(q) \leq -\frac{h(I)}{d(I) r e^r} \frac{(\delta-D-1)^r}{(\delta+r-1)^{r-1}}+\epsilon.$$
If we further assume that $\delta \geq 2D+r+1$ and $\frac{h(I)}{d(I)}\ge r(4e)^r\frac{2\epsilon}{\delta}$, then we see that 
$$h(x) \geq \frac{h(I)}{d(I)} (4e)^{-r}.$$
Therefore, $q$ must vanish at all points in $Z(I)$ with height less than $\frac{h(I)}{d(I)} (4e)^{-r}$ as claimed.
\end{proof}

\section{The quantitative fundamental inequality}\label{sec: DPhtrick}
The goal of this section is to prove our quantitative fundamental inequality in Theorem \ref{quantitative Zhang}. 
In fact we will prove a version for arbitrary polarized endomorphisms. 
First we introduce various definitions for canonical heights and establish some preliminary lemmata. 

\subsection{Canonical heights}\label{canonical heights}
Recall our convention that $K_0$ is either $\bQ$ or a fixed base function field $k(B)$ with an ample line bundle $\mathcal{M}$ on $B$ and all heights are computed with respect to the places of $M_{K_0}$; so $h\coloneqq h_{K_0}$.

Let $\phi\colon Y \to Y$ be an endomorphism of a projective variety $Y$ over $\overline{K_0}$, and let $L$ be a very ample line bundle on $Y$ such that 
\[
\phi^{*}L \simeq L^{\otimes d}, \qquad d \ge 2.
\]
We call the triple $(\phi,Y,L)$ a \emph{polarized dynamical system} of degree $d$ over $\overline{K_0}$. 
Following Zhang \cite{Zha06}, and later Gubler \cite{Gub08} in the function field case, one associates to every polarized dynamical system a canonical adelic semipositive metrized line bundle $\overline{L}_{\phi}$, called the \emph{$\phi$-invariant adelic metrized line bundle}. 
This metric is constructed via a Tate--telescoping procedure and its associated height, as defined in~\eqref{height def}, agrees with the Call--Silverman canonical height \cite{CS93} on points. 
For every $x \in Y(\overline{K_0})$,
\begin{align}\label{cs height}
h_{\overline{L}_{\phi}}(x)
  = \hat{h}_{\phi,L}(x)
  = \lim_{n\to\infty}\frac{h_L(\phi^{n}(x))}{d^{n}}.
\end{align}

In this section we will show that this equality extends from points to higher-dimensional subvarieties, and we will establish several key properties of the resulting canonical height for subvarieties. We begin with the following comparison between the `invariant height' and the `standard height' for endomorphisms of the projective plane. 
The key input is Ingram's explicit bounds of the difference between the local Weil and canonical heights of points from \cite{Ing22}. 
Qualitative versions of this result are also established by Gauthier--Vigny \cite{GV24}, and it's their proof we shall follow.

\begin{proposition}\label{height diff}
Let $\Phi:\mathbb{P}^N\to \mathbb{P}^N$ be a degree $d\ge 2$ endomorphism defined over $\overline{K_0}$ and let $X$ be an irreducible subvariety of $\mathbb{P}_{\overline{K_0}}^N$. 
Let $c_1(N,d)= (N+1)d^N+1$ and $c_2(N,d) = (N+1)(d+1)^N(d^N + 1)^{(N+1)(d+2)^N}$, if $K_0=\bQ$ and $0$ if $K_0$ is a function field. 
We have 
$$|\hat{h}_{\overline{L}_{\Phi}}(X)-\hat{h}_{\mathrm{st}}(X)|\le (\dim X+1)(c_1(N,d)h(\Phi)+c_2(N,d)).$$ 
\end{proposition}

\begin{proof}
Suppose that $\phi$ and $X\subset \bP^N$ are defined over a finite extension $K$ of $K_0$. We may assume that $K$ is the field of definition of $X$ so that the Galois orbit of $X$ has size $[K:K_0]$. 
Let $n_v=[K_v: (K_0)_{v|_{K_0}}]/[K:K_0]$ denote the local degrees such that the $\prod_{v\in M_K}|x|_v^{n_v}=1$ for $x\in K^*$. Let $\tilde{\Phi}: \bb{A}^{N+1} \to \bb{A}^{N+1}$ be a homogeneous lift of $\Phi$, that is $\tilde{\Phi}$ is given by homogeneous polynomials $(F_0,\ldots,F_N)$ of degree $d$ such that it induces $\Phi=[F_0:\cdots:F_N]$ on $\bb{P}^N$. 
Recall that the Green function of $\tilde{\Phi}$ is   
$$G_{\tilde{\Phi},v}(x)= \lim_{n \to \infty} \frac{\log \|\tilde{\Phi}^n(x)\|_v}{d^n},$$
where $$\|x\|_v\coloneqq \|(x_0,\ldots,x_N)\|_v = \max\{|x_0|_v,\ldots,|x_N|_v\},$$ for $x=(x_0,\ldots,x_N) \in \bb{A}^{N+1}(\bb{C}_v)$; see \cite[Section 3.5]{Sil07}. 
Recall that by \cite[Corollary 3.8]{Gub08} (for function fields), and \cite[Proposition 3.2.2]{BGS94} (for number fields) we have that 
\begin{align}\label{intersection comparison}
\left|\overline{L}_{\Phi}|_{X_{\mathrm{gal}}}^{\cdot \dim X+1}- \overline{O(1)}|_{X_{\mathrm{gal}}}^{\cdot \dim X+1}\right|\le [K:K_0]\deg(X) (\dim X+1) d(\Phi,\mathrm{st}),
\end{align}
where 
\begin{align*}
d_v(\tilde{\Phi},\mathrm{st})\coloneqq  \sup_{[x_0:\cdots:x_N]\in\bP^N(\Kbar)}(|G_{\tilde{\phi},v}(x_0,\ldots,x_N)- \log\|(x_0,\ldots,x_N)\|_v|),
\end{align*}
for each $v\in M_K$, and 
\begin{align*}
d(\tilde{\Phi},\mathrm{st})= \displaystyle\sum_{v\in M_K}n_v d_v(\tilde{\Phi},\mathrm{st}).
\end{align*}
By \cite[Lemma 6]{Ing22} for each $v\in M_K$, we have explicit constants $c_{3,v}$ and $c_{4,v}$ such that 
$$\frac{-c_{3,v} - \lambda_{\Hom_d^n,v}(\Phi)}{d-1} \leq G_{\tilde{\Phi},v}(x_0,\ldots,x_n) - \log \|x\|_v  - \frac{1}{d-1} \log \|\tilde{\Phi}\|_v \leq \frac{c_{4,v}}{d-1},$$
where 
$$\lambda_{\Hom_d^n,v}(\Phi) = - \log |\Res(\tilde{\Phi})|_v + (N+1)d^N \log \|\tilde{\Phi}\|_v.$$
Moreover, we have 
\begin{align*}
\begin{split}
\sum_{v\in M_K}n_v c_{3,v} &\le c_3\coloneqq  h_{K_0}(N+1)+Nh_{K_0}(N(d-1)+1)+ (d+1)^Nh_{K_0}(d^N(N+1))+c_5, \\ 
\sum_{v\in M_K}n_v c_{4,v}&\le c_4\coloneqq N h_{K_0}(d+1),
\end{split}
\end{align*}
where $c_5$ is the sum over all (normalized) places of the constants denoted by $c_1$ in \cite[Lemma 4]{Ing22}. 
Note here that inspecting Ingram's definitions of the constants, we have that if $K_0$ is a function field, then we may take $c_3 = c_4 = 0$; otherwise, we may bound 
\begin{align}\label{bound ingram}
\max\{c_3,c_4\}\le (N+1)(d+1)^N (d^N+1)^{(N+1)(d+2)^N}.
\end{align}
Using the product formula we also have  
\begin{align}
\sum_{v\in M_K} n_v\lambda_{\Hom_d^n,v}(\Phi) = (N+1)d^N h_{K_0}(\Phi),
\end{align}
so that 
$$d(\tilde{\Phi},\mathrm{st}) \leq ((N+1)d^N + 1)h_{K_0}(\Phi) + \frac{\max\{c_3,c_4\}}{d-1}.$$
Combining this with \eqref{bound ingram}, the conclusion then follows from \eqref{intersection comparison}.
\end{proof}

Inspired by the construction of canonical heights for subvarieties of abelian varieties due to David--Philippon \cite{DP02}, and by Hutz’s dynamical canonical height for subvarieties \cite{Hut18}, we now introduce the canonical height of an arbitrary subvariety with respect to the polarized dynamical system. This provides a natural extension of the Call--Silverman height \eqref{cs height} from points to higher-dimensional subvarieties and will be crucial for our proofs throughout the rest of this paper. 

\begin{defprop}\label{canonical height def/prop}
Let $(\phi,Y,L)$ be a polarized dynamical system of degree $d\ge 2$ over $\overline{K_0}$.
Let $X\subset Y$ be an irreducible subvariety and let $\iota: Y\hookrightarrow \bP^N$ be an embedding defined by a complete linear system of $L$. The canonical height of $X$ with respect to $(\phi,Y,L)$ is then defined to be 
\begin{equation}\label{DPhstyle}
\hat{h}_{\phi,L}(X)\coloneqq \lim_{n\to\infty} \frac{\hat{h}_{\mathrm{st},\bP^N}\bigl(\iota(\phi^n(X))\bigr)}{d^n}.
\end{equation}
\begin{enumerate}
\item The limit in \eqref{DPhstyle} exists and is independent of the choice of embedding $\iota$.
\item We have
$\hat{h}_{\phi,L}(X)
=\lim_{n\to\infty} \frac{\hat{h}_{\mathrm{st},\bP^N}\bigl(\iota(\phi^n(X))\bigr)}{d^n}
=\lim_{n\to\infty} \frac{\hat{h}_{\mathrm{Ph},\bP^N}\bigl(\iota(\phi^n(X))\bigr)}{d^n}.$
\item $\hat{h}_{\phi,L}(X)=\hat{h}_{\overline{L}_{\phi}}(X)$.
\end{enumerate}
\end{defprop}

\begin{proof}
(1) For existence of the limit in \eqref{DPhstyle}, choose an embedding $\iota_0:Y\hookrightarrow\bP^N$ as in \cite[Corollary~2.2]{Fak03}, so that there is an endomorphism $\psi:\bP^N\to\bP^N$ of degree $d$ with
$\psi\circ\iota_0=\iota_0\circ\phi.$
Then
\[
\hat{h}_{\phi,L}(X)
=\lim_{n\to\infty}\frac{\hat{h}_{\mathrm{st},\bP^N}\bigl(\iota_0(\phi^n(X))\bigr)}{d^n}
=\lim_{n\to\infty}\frac{\hat{h}_{\mathrm{st},\bP^N}\bigl(\psi^n(\iota_0(X))\bigr)}{d^n},
\]
and the last limit exists by Hutz's construction of the canonical height for $\psi$ on $\bP^N$ (see \cite{Hut09}). This shows existence of the limit for the embedding $\iota_0$.

Now let $\iota:Y\hookrightarrow\bP^N$ be any other embedding given by the complete linear system of $L$. Let 
$\alpha_{\iota},\alpha_{\iota_0}:\iota^*O(1)\xrightarrow{\sim}L,\ \ \iota_0^*O(1)\xrightarrow{\sim}L$
be the canonical isomorphisms. Consider the two adelic metrized line bundles on $Y$, both supported on $L$, given by
\[
\overline{L}_1\coloneqq (\alpha_{\iota})_*\iota^*(\overline{O(1)}),\qquad
\overline{L}_2\coloneqq (\alpha_{\iota_0})_*\iota_0^*(\overline{O(1)}),
\]
 By \cite[Corollary~3.8]{Gub08} in the function field case and \cite[Proposition~3.2.2]{BGS94} in the number field case, there exists a constant $C\ge 0$, depending only on the two metrics, such that for every irreducible subvariety $Z\subset Y$,
\[
\bigl|\hat{h}_{\mathrm{st},\bP^N}(\iota(Z))-\hat{h}_{\mathrm{st},\bP^N}(\iota_0(Z))\bigr|
\le C\,(\dim Z+1).
\]
Applying this with $Z=\phi^n(X)$, we obtain
\[
\left|
\frac{\hat{h}_{\mathrm{st},\bP^N}\bigl(\iota(\phi^n(X))\bigr)}{d^n}
-
\frac{\hat{h}_{\mathrm{st},\bP^N}\bigl(\iota_0(\phi^n(X))\bigr)}{d^n}
\right|
\le \frac{C(\dim X+1)}{d^n}.
\]
The right-hand side tends to $0$ as $n\to\infty$, so the sequence defining $\hat{h}_{\phi,L}(X)$ has the same limit for $\iota$ as for $\iota_0$. Thus the limit in \eqref{DPhstyle} exists and is independent of the embedding.

(2) This follows from \eqref{nf heights compare} applied to $\iota(\phi^n(X))$ for each $n$, as before.

(3) We keep the embedding $\iota_0:Y\hookrightarrow\bP^N$ and endomorphism $\psi:\bP^N\to\bP^N$ as in the proof of (1). Let $\overline{O(1)_{\psi}}$ denote the $\psi$-invariant metrized line bundle on $\bP^N$. By functoriality of the construction, we have 
$\iota_0^*\overline{O(1)}_{\psi}=\overline{L}_{\phi}.$ By Proposition~\ref{height diff}, there exists a constant $C'\ge 0$ such that for every irreducible subvariety $Z\subset\bP^N$,
\[
\bigl|\hat{h}_{\overline{O(1)}_{\psi}}(Z)-\hat{h}_{\mathrm{st},\bP^N}(Z)\bigr|
\le C'(\dim Z+1).
\]
Applying this to $Z=\psi^n(\iota_0(X))=\iota_0(\phi^n(X))$, dividing by $d^n$ and letting $n\to\infty$, we get
\[
\hat{h}_{\phi,L}(X)
=\lim_{n\to\infty}\frac{\hat{h}_{\mathrm{st},\bP^N}\bigl(\psi^n(\iota_0(X))\bigr)}{d^n}
=\lim_{n\to\infty}\frac{\hat{h}_{\overline{O(1)}_{\psi}}\bigl(\psi^n(\iota_0(X))\bigr)}{d^n}.
\]
By construction of the invariant metric, we have the scaling relation
$$\hat{h}_{\overline{O(1)}_{\psi}}\bigl(\psi(Z)\bigr)
=d\,\hat{h}_{\overline{O(1)}_{\psi}}(Z),$$
for every subvariety $Z\subset\bP^N$, 
so we infer
$
\hat{h}_{\phi,L}(X)
=\hat{h}_{\overline{O(1)}_{\psi}}\bigl(\iota_0(X)\bigr).
$
Finally, since $\iota_0^*\overline{O(1)}_{\psi}=\overline{L}_{\phi}$ and heights are functorial, we have
$\hat{h}_{\overline{L}_{\phi}}(X)
=\hat{h}_{\overline{\mathcal{O}(1)}_{\psi}}\bigl(\iota_0(X)\bigr).$
The result follows. 
\end{proof}


Recall that for an adelic metrized line bundle $\overline{L}$ on $X$, Zhang's essential minimum is
\[
e(\overline{L},X)
\;=\;
\sup_{U\subseteq X}\;\inf_{P\in U(\overline{K_0})} h_{\overline{L}}(P),
\]
where $U$ ranges over all nonempty Zariski open subsets of $X$. 
For the next result we will need the following lemma regarding the standard metrized line bundle. 

\begin{lemma}\label{bezout}
  Let $K_0$ be a function field. For every irreducible subvariety $X\subset \bP^N_{K_0}$ we have 
  $$\hat{h}_{\mathrm{st},\mathbb{P}^N}(X)\ge e(\overline{O(1)}_{\mathrm{st}},X).$$
\end{lemma}

\begin{remark}
   It is tempting to invoke Gubler’s lower bound in the fundamental inequality \cite[Proposition 4.2]{Gub07b} in order to deduce Lemma \ref{bezout} directly.
However, one must resist this temptation as demonstrated by the recent counterexample of Guo–Yuan \cite{GY25}
\end{remark}

\begin{proof}
Recall that $K_0=k(B),$ where $B$ has dimension $\delta>0$ and we have fixed a very ample line bundle $\cal{M}$ on $B$. 
Let $r:=\dim X \ge 1$ and let $\calX\subset \bP^N_B$ denote the Zariski closure of $X$ in $\bP^N_B$. Because $k$ is infinite, we may choose $r$ hyperplanes $H_1,\dots,H_r \subset \bP^N_k$
defined over $k$ and in general position, in the following sense:
\begin{itemize}
  \item the intersection $X \cap H_1 \cap \cdots \cap H_r$ is proper in $\bP^N_{K_0}$ and consists of
  $e = \deg(X)$
  points $p_1,\ldots,p_e$ (with multiplicity);
  \item the total intersection
  $\mathcal{X} \cap \mathcal{H}_1 \cap \cdots \cap \calH_r \subset \bP^N_B,$
  where $\calH_i := H_i \times B$, is proper.
\end{itemize}
Let $\pi:\bP^N_B\to B$ the natural projection. 
Each $\calH_i$ has numerical class $\calL$, so in the Chow ring of $\bP^N_B$ we have
\begin{equation}\label{eq:int-equality}
  \bigl(\calX \cap \calH_1 \cap \cdots \cap \calH_r\bigr)
  \cdot \calL \cdot \pi^*\calM^{\,\delta-1}
  =
  \calX \cdot \calL^{r+1} \cdot \pi^*\calM^{\,\delta-1}
  =
  h_{\mathrm{st}}(X).
\end{equation}
We may write 
\[
  \calZ := \calX \cap \calH_1 \cap \cdots \cap \calH_r= \sum_{i=1}^e \Gamma_{p_i} + V,
\]
where each $\Gamma_{p_i}$ is the Zariski closure of $p_i\in X\cap H_1\cap\cdots\cap H_r$ in $\bb{P}^N_B$, and $V$ is an effective cycle supported over a proper closed subset of $B$. Since $V$ is effective and $\calL, \calM$ are ample, we have 
\begin{equation}\label{eq:sum-ineq}
  \sum_{i=1}^e h_{\mathrm{st}}(p_i)= \sum_{i=1}^e
    \bigl(\Gamma_{p_i} \cdot \calL \cdot \pi^*\calM^{\,\delta-1}\bigr)
  \;\le\;
  \calZ \cdot \calL \cdot \pi^*\calM^{\,\delta-1}=h_{\mathrm{st}}(X),
\end{equation}
where the equation on the right hand side follows by \eqref{eq:int-equality}. 
In particular,
\[
  \min_{1 \le i \le e} h_{\mathrm{st}}(p_i)
  \;\le\;
  \frac{1}{e}\sum_{i=1}^e h_{\mathrm{st}}(p_i)
  \;\le\;
  \frac{h_{\mathrm{st}}(X)}{e}
  =
  \hat{h}_{\mathrm{st}}(X).
\]
As we vary the $r$-tuple $(H_1,\dots,H_r)$ in a Zariski dense open subset of the dual projective space
$(\bP^{N,\vee}_k)^r$, we obtain a Zariski dense set of points $p_i \in X(\overline{K_0})$ with
\[
  h_{\mathrm{st}}(p_i) \le \frac{h_{\mathrm{st}}(X)}{\deg(X)}.
\]
The conclusion follows. Note that $r=0$ the statement is trivial. 
\end{proof}

For endomorphisms $\Phi$ of $\bP^N$ we will write $\hat{h}_{\Phi}\coloneqq \hat{h}_{\Phi,O(1)}$.
Finally, we will need the following.

\begin{proposition}\label{height and sup}
Let $\Phi:\mathbb{P}^N\to \mathbb{P}^N$ be an endomorphism of degree $d\ge 2$ defined over $\overline{K_0}$, and let $X\subset \mathbb{P}^N_{\overline{K_0}}$ be an irreducible subvariety. Then
\[
\hat{h}_{\Phi}(X)\;\le\;(\dim X+1)\left(\hat{h}_{\mathrm{st},\mathbb{P}^N}(X)
\;+\;\sup_{x\in X(\overline{K_0})}\bigl|\hat{h}_{\Phi}(x)-h(x)\bigr|\right).
\]
\end{proposition}

\begin{proof}
Let $\overline{O(1)}$ denote the standard adelic metrized line bundle on $\mathbb{P}^N$, and let $\overline{O(1)}_{\Phi}$ be the $\Phi$-invariant adelic metrized line bundle. By construction, for every $x\in \bP^N(\overline{K_0})$ we have
$h_{\overline{O(1)}_{\Phi}}(x)=\hat{h}_{\Phi}(x)$
and 
$h_{\overline{O(1)}}(x)=h(x).$
Set $M \coloneqq  \sup_{x\in X(\overline{K_0})}|\hat{h}_{\Phi}(x)-h(x)|$. 
Then for every nonempty Zariski open subset $U\subseteq X$ we obtain
$$\inf_{x\in U} h_{\overline{O(1)}}(x)-M
\le
\inf_{x\in U} h_{\overline{O(1)}_{\Phi}}(x)
\le
\inf_{x\in U} h_{\overline{O(1)}}(x)+M.$$
Taking the supremum over all such $U$ gives
\begin{equation}\label{eq:ess-comp}
e(\overline{O(1)},X) - M
\;\le\;
e(\overline{O(1)}_{\Phi},X)
\;\le\;
e(\overline{O(1)},X) + M.
\end{equation}
 Zhang's fundamental inequality in the number field case \cite{Zha98} and its function field analogue due to Gubler \cite[Lemma~4.1]{Gub07b}, applied to $\overline{O(1)}_{\Phi}$, yield
\begin{equation}\label{eq:fund-phi}
\hat{h}_{\overline{O(1)}_{\Phi}}(X)
\;\le\;
(\dim X+1)\,e(\overline{O(1)}_{\Phi},X).
\end{equation}
Combining \eqref{eq:fund-phi} with the upper bound in \eqref{eq:ess-comp}, and recalling  the equality of heights in 
Definition/Proposition~\ref{canonical height def/prop}(3), we obtain
\[
\hat{h}_{\Phi}(X)
=
\hat{h}_{\overline{O(1)}_{\Phi}}(X)
\le
(\dim X+1)\,e(\overline{O(1)}_{\Phi},X)
\le
(\dim X+1)\bigl(e(\overline{O(1)},X)+M\bigr).
\]
The conclusion follows using Lemma \ref{bezout} to bound $e(\overline{O(1)},X)$ from above in the case of function fields or Zhang's lower bound in the Theorem of Successive minima \cite{Zha98} in the case of number fields.  
\end{proof}


\subsection{A generic curve and height preservation}\label{sec: generic}
In this section we revisit Gubler's construction of a `transcendental generic curve' \cite[\S 3.11, \S 3.12]{Gub08} and establish the key height comparison which we will use in reducing our results from arbitrary function fields to function fields of curves in  Proposition \ref{prop: generic}.
We thank Xinyi Yuan for pointing out inaccuracies in the literature and for his feedback on a previous version.
\par 
Let $B$ be a geometrically integral and normal projective variety of dimension $\delta\ge 2$ over an algebraically closed field $k$. Let $\cal{M}$ be a fixed very ample line bundle on $B$ inducing an embedding 
$B\hookrightarrow\bb{P}^M$ and write $K=k(B)$. 
Let $\zeta_j^{(i)}$ for $i \in \{0,\ldots,\delta-1\}$ and $j \in\{ 0,\ldots,M\}$ be algebraically independent variables (over $k$) and let 
$$k' = k(\zeta_j^{(i)}: i\in\{0,\ldots,\delta-1\}, j\in\{1,\ldots,M\}).$$
We write $B':=B_{k'}=B\times_kk'$, and consider the generic linear projection (over $k$)
\[
\begin{array}{rcl}
 B' & \dashrightarrow & \bb{P}_{k'}^{\delta-1} \\[4pt]\left[ x_0 : \cdots : x_M \right] & \longmapsto &
\left[
\zeta_0^{(0)} x_0 + \cdots + \zeta_M^{(0)} x_M :
\cdots :
\zeta_0^{(\delta-1)} x_0 + \cdots + \zeta_M^{(\delta-1)} x_M
\right].
\end{array}
\]
Let $Y$ be its indeterminacy locus, which is the intersection of $B'$ with $\delta$ generic hyperplanes. 
A standard application of Bertini's theorem tells us that $Y$ is a closed point of $B'$.  
Let $\widetilde{X}$ be the blow-up of $B'$ along the finite subscheme $Y$, and let
$\beta : \widetilde{X} \longrightarrow B'$
be the blow-up morphism. This step is ommitted in \cite{Gub08}.
By construction, the rational map given by the generic linear projection extends to a morphism
$\widetilde{\pi} : \widetilde{X} \longrightarrow \mathbb{P}^{\delta-1}_{k'}.$
Now let $X'$ be the normalization of $\widetilde{X}$, and let
$\nu : X' \longrightarrow \widetilde{X}$
be the finite birational normalization map. We obtain a dominant morphism
\[\pi := \widetilde{\pi} \circ \nu : X' \longrightarrow \mathbb{P}^{\delta-1}_{k'}.
\]
Note that $X'$ is by construction a normal projective variety over $k'$ with dimension $\delta$, as it is the normalization of the blow-up of $B'$. Observe that since $B'$ is normal, the morphism $\phi: X' \to B'$ is an isomorphism on $B' \setminus Y$. In particular, all divisors of $B'$ can be identified as divisors of $X'$ canonically and any new divisors of $X'$ must contract to $Y$ under $\phi$.  
\par 
Denote by $\eta$ the generic point of $\mathbb{P}^{\delta-1}_{k'}$, and set
\[
k''\coloneqq \kappa(\eta) = k'(\mathbb{P}^{\delta-1}).
\]
\begin{proposition} \label{GublerGeneric1}
The generic fiber  $B''\coloneqq X'_\eta = X' \times_{\mathbb{P}^{\delta-1}_{k'}} \mathrm{Spec} k''$ of $\pi$
is a normal projective curve over $k''$. 
Moreover, we have equality of function fields $k''(B'')=k'(X')=k'(B')$. 
\end{proposition}
   
\begin{proof}
First, $\pi$ is a projective morphism since $X'$ is projective over $k'$. Thus the generic fiber $B''$ is projective over $k'$. Next we compute the function field of the generic fiber. 
Choose an affine open neighborhood
$U = \mathrm{Spec}(A) \subset \mathbb{P}^{\delta-1}_{k'}$ of $\eta$, so that
$k'' = \mathrm{Frac}(A)$. 
Since $\pi$ is dominant, there exists a nonempty affine open subset 
$V = \mathrm{Spec}(A') \subset X'$ such that $\pi(V) \subset U$ and the induced map of rings
$A \to A'$ makes $A'$ into an integral $A$-algebra with $\mathrm{Frac}(A') = k'(X')$. 
The part of the generic fiber lying over $V$ is
$V_\eta 
  = V \times_U \mathrm{Spec}\, k''
  = \mathrm{Spec}\,(A' \otimes_A k'')
  \cong \mathrm{Spec}\,(S^{-1}A'),$
where $S = A \setminus \{0\} \subset A$ and we use that $k'' = \mathrm{Frac}(A) = S^{-1}A$.
Thus $A' \otimes_A k'' \cong S^{-1}A'$ and
$\mathrm{Frac}(S^{-1}A') = \mathrm{Frac}(A') = k'(X').$
Since the generic point of $B''$ lies in $V_\eta$, this shows that the function field
of $B''$ is $\mathrm{Frac}(A') = k'(X')$, i.e.
$k''(B'') = k'(X').$
Finally, the blow-up $\widetilde{X} \to B'$ is birational and $X'$ is its normalization, so
$k'(X') = k'(B')$.

Next, we determine the dimension of $B''$. Since $X'$ is an integral projective variety of
dimension $\delta$ over $k'$, we have
$\dim X' = \delta = \operatorname{tr.deg}_{k'} k'(X').$ 
On the other hand,
$\mathrm{tr.deg}_{k'} k'' = \dim \mathbb{P}^{\delta-1}_{k'} = \delta-1.$
The morphism $\pi$ is dominant, so there is an inclusion of fields
$k'' = k'(\mathbb{P}^{\delta-1}) \hookrightarrow k'(X') = k''(B'').$
Hence
$\mathrm{tr.deg}_{k''} k''(B'')= \mathrm{tr.deg}_{k'} k'(X') - \mathrm{tr.deg}_{k'} k''
= 1.$
Therefore $B''$ is a projective integral curve over $k''$.

It remains to show that $B''$ is normal (and hence regular). 
Cover $B''$ by affine open subsets of the form
$V_\eta = \mathrm{Spec}(S^{-1}A')$
constructed above. It suffices to show that $V_{\eta}$ is normal. 
Since $X'$ is normal, the domain $A'$ is integrally closed in its field of fractions
$k'(X')$, and therefore each localization $S^{-1}A'$ is also integrally closed in 
$k'(X') = k''(B'')$. 
Thus every affine open $V_\eta$ is normal, and $B''$ is a normal (Noetherian) integral
curve over $k''$.
\end{proof} 

We can now state the main result in this section.

\begin{proposition} \label{prop: generic}
Let $n\in\bN$. Let $V \subseteq \bb{P}^n_{k(B)}$ be an irreducible subvariety. Then we have 
$h_{k(B)}(V) = h_{k''(B'')}(V).$
Furthermore if $\Phi: \bb{P}^n_{k(B)} \to \bb{P}^n_{k(B)}$ is an endomorphism of degree $d \geq 2$, then 
$$\hat{h}_{\Phi,k(B)}(V) = \hat{h}_{\Phi,k''(B'')}(V).$$
\end{proposition}

First, we will focus on divisors of $X'$ that come from divisors of $B$.

\begin{lemma} \label{GublerGeneric2}
Let $Z$ be a prime divisor on $B$, defined over $k$, and write
$Z' := Z \otimes_k k' \subset B'$. Let $\widetilde{Z} \subset X'$ denote the
strict transform of $Z'$ in $X'$. Then:
\begin{enumerate}
\item $\widetilde{Z}$ is a horizontal prime divisor for $\pi$, i.e.,
$\pi|_{\widetilde{Z}} : \widetilde{Z} \to \mathbb{P}^{\delta-1}_{k'}$ is dominant.
In particular, there exists a unique closed point $p_{\widetilde{Z}} \in B''$
such that 
\[
\mathcal{O}_{X',\widetilde{Z}}=\mathcal{O}_{B'',p_{\widetilde{Z}}},
\]
viewed as subrings of the common function field
$k'(X') = k''(B'')$.

\item We have 
$\deg_{k''}(p_{\widetilde{Z}}) = \deg_{\mathcal{M}}(Z),$
where $\deg_{k''}(p_{\widetilde{Z}}) := [k''(p_{\widetilde{Z}}):k'']$ is the degree
of the closed point $p_{\widetilde{Z}}$ over $k''$. 
\end{enumerate}
\end{lemma}

\begin{proof}
For (1), recall that we have fixed an embedding $B \hookrightarrow \mathbb{P}^N_k$. Consider
linear projections $\lambda: \mathbb{P}^N_k \dashrightarrow \mathbb{P}^{\delta-1}_k$
with center disjoint from $Z$. For such a projection $\lambda$, the restriction
$\lambda|_Z : Z \rightarrow \mathbb{P}^{\delta-1}_k$
is a morphism between integral varieties of the same dimension $\delta-1$.
For a general choice of linear projection, $\lambda|_Z$ is finite, hence dominant. After base change to $k'$, 
the restriction of $\pi$ to the strict transform $\widetilde{Z}$ is birational
to $\lambda|_Z$ over the function field, and in particular is dominant.
Thus $\widetilde{Z}$ is horizontal and its generic point lies over~$\eta$.

We now construct $p_{\widetilde{Z}}$ and identify the local rings. Let
$K := k'(X') = k''(B'')$ be the common function field. The prime divisor
$\widetilde{Z} \subset X'$ defines a discrete valuation $v_{\widetilde{Z}}$ on $K$
with valuation ring $\mathcal O_{X',\widetilde{Z}}$. Since $\widetilde{Z}$ is
horizontal for $\pi$, its generic point maps to $\eta$, so $v_{\widetilde{Z}}$ is
trivial on $k'' = k'(\mathbb P^{\delta-1})$. Because $B''$ is a smooth projective
curve over $k''$ with function field $K$, there is a unique closed point
$p_{\widetilde{Z}} \in B''$ such that the valuation $v_{p_{\widetilde{Z}}}$ on $K$
coincides with $v_{\widetilde{Z}}$; equivalently,
$\mathcal O_{B'',p_{\widetilde{Z}}}$ is the valuation ring corresponding to
$v_{\widetilde{Z}}$. In particular, we get the equality of local rings
$\mathcal{O}_{X',\widetilde{Z}}=\mathcal{O}_{B'',p_{\widetilde{Z}}}.$

For (2), we work with the generically finite morphism
$\pi|_{\widetilde{Z}} : \widetilde{Z} \to \mathbb P^{\delta-1}_{k'}$. Its degree is
\[
\deg(\pi|_{\widetilde{Z}})
  = [k'(\widetilde{Z}) : k'(\mathbb{P}^{\delta-1})]
  = [k''(p_{\widetilde{Z}}) : k''],
\]
where $k'(\widetilde{Z})$ is the function field of $\widetilde{Z}$ and
$k''(p_{\widetilde{Z}})$ is the residue field of $p_{\widetilde{Z}} \in B''$. 
On the other hand, since $\pi$ is the normalization of the generic linear projection with center
disjoint from $Z$, which is defined over $k$, the degree of $\lambda|_{Z}$
equals the degree of $Z$ with respect to the fixed embedding
$B \hookrightarrow \mathbb{P}^N_k$ induced by $\mathcal{M}$; that is,
$\deg(\lambda|_{Z}) = \deg_{\calM}Z$. This degree is unchanged by base change
to $k'$ and by passing to the strict transform, so
$\deg(\pi|_{\widetilde{Z}}) = \deg _{\calM}Z.$ The conclusion follows. 
\end{proof}

Note that the divisors of $X'$ split into two types. 
Those whose image under $$\phi:=\beta\circ\nu: X'\to B'$$ is a prime divisor of $B'$ and the new `exceptional' divisors that map into the indeterminacy point $Y$. 
Note that the non-exceptional prime divisors of $X'$ correspond precisely to the prime divisors of $B'$. 
We now handle the new exceptional divisors.

\begin{lemma} \label{GublerGeneric3}
Let $E$ be an exceptional divisor of $X'$ and $f \in k(B)^{\times}$. Then the valuation corresponding to $E$ satisfies $v_{E}(f)= 0$. 
\end{lemma}

\begin{proof}
Let $Y$ be the indeterminancy point of the generic projection and note that it cannot be contained in any proper closed subset of $B'$ that is defined over $k$. Let  $f \in k(B)^{\times}$, so that its zero and poles loci are defined over $k$. We infer that $f$ is well defined at $Y$, as $Y$ is generic and so is not contained in any closed subvariety defined over $k$; in fact $f\in \mathcal{O}^{\times}_{B',Y}$.  
Let $\eta_E$ denote the generic point of $E$. Since $\phi(\eta_E)=Y$, the dominant map $\phi$ induces an injection of local rings $\mathcal{O}_{B',Y} \xhookrightarrow{} \mathcal{O}_{X',\eta_E}$ such that the maximal ideal of $\cal{O}_{X',\eta_E}$ contracts to the maximal ideal of $\cal{O}_{B',Y}$. Since $f$ is a unit in $\mathcal{O}_{B',Y}$ its image in $\mathcal{O}_{X',\eta_E}$ will also be a unit. The lemma follows. 
\end{proof}

The last lemma we will need is the following. 
\begin{lemma}\label{notrational}
Let $Z'$ be a prime divisor of $B'$ that is not defined over $k$ and let $f\in k(B)\subset k'(B')$. Then $v_{Z'}(f)=0$. 
\end{lemma}

\begin{proof}
This is clear since for $f\in k(B)$ its divisor $\mathrm{div}_B(f)$ of zeros and poles is a linear combination of prime divisors of $B$. 
\end{proof}

We are now ready to prove the key result.

\bigskip

\noindent\emph{Proof of Proposition \ref{prop: generic}:}
We begin by noting that it suffices to prove that the Weil heights agree in the case of points in all projective spaces. Indeed if $V$ is a  irreducible subvariety in some projective space that is defined over $k(B)$, then its Chow form is a $k(B)$-point in another projective space and it suffices to recall from \S\ref{sec: heights} that the height of $V$ is precisely the height of its Chow form. The claim regarding the $f$-canonical heights also follows since they are given as limits of standard heights from Definition/Proposition \ref{height def}. 

So in what follows, we let 
 $x = [f_0:\dots:f_n] \in \mathbb{P}^n\bigl(k(B)\bigr)$ with $f_i \in k(B)$.
We will show that 
$h_{k(B)}(x) \;=\; h_{k''(B'')}(x).$
Recall that $k''(B'') = k'(X')$ by Proposition~\ref{GublerGeneric1}, and that
$B''$ is a smooth projective curve over $k''$. Identifying the places in
$M_{k''(B'')}$ with the closed points $p\in B''$, the height of $x$ over $k''(B'')$
is given by
\begin{equation}\label{eq:h-generic-curve}
h_{k''(B'')}(x)
= -\sum_{p \in B''} [k''(p):k'']\min\{v_p(f_0),\ldots,v_p(f_n)\},
\end{equation}
where $v_p$ is the discrete valuation of $k''(B'')=k'(X')$ associated to $p$.

We now pass from closed points of $B''$ to divisors on $X'$. Let
$j : B'' \to X'$ be the natural map from the generic fiber. For each closed point
$p \in B''$, let $Z_p\subset X'$ be the prime divisor given by the Zariski closure of
$\{j(p)\}$ in $X'$. The map $j$ induces a local homomorphism $\mathcal{O}_{X',j(p)} \hookrightarrow \mathcal{O}_{B'',p}$, which composed with the inclusion of $j(p)$ in its Zariski closure gives an inclusion $\mathcal{O}_{X',Z_p} \hookrightarrow \mathcal{O}_{B'',p}.$
We view this as an inclusion of subrings of the common function field
$k'(X') = k''(B'')$. Since $X'$ is normal and $Z_p$ has codimension~$1$, the local ring
$\mathcal{O}_{X',Z_p}$ is a discrete valuation ring; likewise, $\mathcal{O}_{B'',p}$
is a discrete valuation ring. Since they also have the same field of fractions and one contains the other, they must coincide:
$\mathcal{O}_{X',Z_p} = \mathcal{O}_{B'',p}.$
In particular, the valuation on $k'(X')=k''(B'')$ induced by $Z_p$ agrees with the valuation associated to $p\in B''$. 

Conversely, every horizontal prime divisor $Z' \subset X'$ determines a unique
closed point $p_{Z'}\in B''$ with $\mathcal O_{X',Z'} = \mathcal O_{B'',p_{Z'}}$.
Thus the nontrivial contributions in \eqref{eq:h-generic-curve} can be indexed by
horizontal prime divisors $Z'\subset X'$ and we have 
\begin{equation}\label{eq:h-generic-curve-horizontal}
h_{k''(B'')}(x)
=-\sum_{\substack{Z' \subset X'\\ \text{horizontal prime}}}
   [k''(p_{Z'}):k''] \min\{v_{Z'}(f_0),\ldots,v_{Z'}(f_n)\}.
\end{equation}

Moreover, the places of $k'(X')$ are given by the order of vanishing
$v_{Z'}$ at prime divisors $Z' \subset X'$. By Lemma~\ref{GublerGeneric3},
since each $f_i\in k(B)\subset k'(X')$, we have $v_E(f_i)=0$ for every exceptional
divisor $E\subset X'$. Hence only the non-exceptional divisors of $X'$ may
contribute to the height $h_{k'(X')}(x)$ or to $h_{k''(B'')}(x)$, and these are precisely the strict
transforms of prime divisors on $B'$. 
Moreover, by Lemma~\ref{notrational}, if $Z' \subset B'$ is a prime divisor not defined
over $k$, then for $f\in k(B)$ we have $v_{Z'}(f)=0$; the same holds for the strict
transform of $Z'$ on $X'$. Thus only those prime divisors of $B'$ coming from base
change of a divisor $Z\subset B$ defined over $k$ can contribute. 
In particular,
\begin{align}
\begin{split}
h_{k''(B'')}(x)
&= -\sum_{\substack{Z \subset B\\ \text{horizontal prime over $k$}}}
   [k''(p_{\widetilde{Z}}):k''] \min\{v_{\widetilde{Z}}(f_0),\ldots,v_{Z'}(f_n)\},\text{ and }\\
h_{k'(X')}(x) &= h_{k'(B')}(x)= -\sum_{\substack{Z \subset B \text{ prime }\\ \text{ over } k}}
    \deg_{\mathcal{M}}(Z)\,
    \min\{v_{Z \otimes_k k'}(f_0),\ldots,v_{Z \otimes_k k'}(f_n)\}.
\end{split}
\end{align}
It now follows from Lemma \ref{GublerGeneric2} that $h_{k''(B'')}(x)=h_{k'(B')}(x)$. The Proposition then follows since $k'(B')= k(B)\otimes_k k'$ and extensions of the field of constants don't affect heights. 
\qed

\begin{remark}\label{compare with XY}
 We note that a construction of a similar ``generic curve'' also appeared later in the work of Xie--Yuan on the geometric Bogomolov conjecture \cite{XY22}. There is, however, a difference between the two settings. In Gubler's approach, the generic curve is obtained after adjoining a transcendental extension $k''$ of the base field $k$ which is algebraically independent from $k(B)$. In contrast, Xie--Yuan seek a transcendental extension lying inside $k(B)$, since their arguments require avoiding any extension of the ground function field $K=k(B)$. 
Note that with Xie--Yuan's approach the height preservation of Proposition \ref{prop: generic} may fail for points in $K$; in general we get an inequality \cite[Lemma 3.7]{XY22}.
\end{remark}

\subsection{Proof of the dynamical quantitative fundamental inequality}

We let $Y$ denote a projective variety defined over the algebraic closure of $K_0$ and let $L$ be a very ample line bundle on $Y$. 
For $d\ge 2$, we say that the pair $(Y,L)$ is \textbf{$d$-good} if for any embedding $\iota: Y \xhookrightarrow{} \bb{P}^N$ given by a complete linear system of $L$, and for all $n\ge 0$, the map  $H^0(\bb{P}^N,O(n)) \xrightarrow{\iota^{*}} H^0(Y,L^{\otimes n})$
is surjective, and the image $\iota(Y)$ is cut out set-theoretically by equations of degree $\leq d$.

Note that replacing $L$ by a power we can always ensure that the pair $(Y,L)$ is $d$-good \cite[Corollary 2.2]{Fak03}, \cite[Theorem 1, Theorem 3]{Mum70}.
Recall our notations from \S\ref{sec: heights} and \S\ref{canonical heights}. 
Our aim is to prove the following result and then deduce Theorem \ref{quantitative Zhang}.

\begin{theorem}[Quantitative fundamental inequality]\label{quantitative Zhang2}
Let $K_0$ be either $\bQ$ or a function field. 
Let $(\phi, Y, L)$ be a degree $d\ge 2$ polarized dynamical system over $\overline{K_0}$ and assume that the pair $(Y,L)$ is $d$-good. 
Fix an embedding $\iota: Y\hookrightarrow \bP^N$ given by a complete linear system of $L$ and let 
\begin{align*}
C(\phi,\iota) =
\begin{cases}
\displaystyle 
\sup_{x\in Y(\overline{K_0})} \bigl| \hat{h}_{\phi}(x) - h(x) \bigr|
 \text{ \hspace{30mm}    , if } Y \subseteq \bb{P}^N \text{ and }(Y,L,\iota) = (Y, O(1), \mathrm{id}) \\[1.2em]
\displaystyle 
\sup_{x\in Y(\overline{K_0})} \bigl| \hat{h}_{\phi,L}(x) - h(\iota(x)) \bigr|
+ 2h((N+1)!) + 2h(N+1)
 \text{\hspace{2mm}, otherwise.}
\end{cases}
\end{align*}
If $X'\subset Y$ is an irreducible subvariety, then the set 
\begin{align*}
\left\{x\in X'(\overline{K_0})~:~\hat{h}_{\phi,L}(x)<\frac{1}{2(\dim X'+1)(4e)^{\dim X'+1}}\hat{h}_{\phi,L}(X')\right\}
\end{align*}
is contained in a proper subvariety $Z\subsetneq Y$ with $\deg_LZ$ at most 
$$3N (\dim X'+1) (\deg_L X')^2 \left(\frac{(2+6h(N+1))(\dim X'+1)^2(4e)^{\dim X'+1}(C(\phi, \iota) + \frac{7}{2} h(N+1))}{\hat{h}_{\phi,L}(X')}+1\right)^{N+1}.$$
\end{theorem}


Before we begin the proof, we state a version of our theorem in the case of number fields or function fields of curves. We will reduce Theorem \ref{quantitative Zhang2} to the following result using the construction of the generic curve from \S\ref{sec: generic}.

\begin{theorem}\label{quantitative Zhang3}
Let $K$ be either $\bQ$ or the function field of a curve. Let $\Phi: \bb{P}^N \to \bb{P}^N$ be an endomorphism defined over $\ovl{K}$, let $X$ be a subvariety of $\bb{P}^N$ and let $K_1$ be a subfield of $\ovl{K}$ such that $\Phi$ and $X$ are defined over $K_1$.    Let $c(\Phi,K_1)$ be a constant such that:
\begin{enumerate}
\item $\hat{h}_{\Phi,K}(\Phi^m(X))\;\le\;(\dim X+1)\left(\hat{h}_K(\Phi^m(X))
 + c(\Phi,K_1) \right)$ for all $m \geq 0$,
\item $\sup_{x \in \Phi^m(X)(K_1)} |\hat{h}_{\Phi,K}(x) - \hat{h}_K(x)| \leq c(\Phi,K_1)$ for all $m \geq 0$.
\end{enumerate}
Then the set 
\begin{align*}
\left\{x\in X(K_1) ~:~\hat{h}_{\Phi,K}(x)<\frac{1}{2(\dim Y+1)(4e)^{\dim Y+1}}\hat{h}_{\Phi,K}(Y)\right\}
\end{align*}
is contained in a proper subvariety $Z\subsetneq X$ with $\deg Z$ at most 
$$3N (\dim X+1) (\deg X)^2 \left(\frac{(2+6h_K(N+1))(\dim X+1)^2(4e)^{\dim X+1}(c(\Phi,K_1) + \frac{7}{2} h_K(N+1))}{\hat{h}_{\Phi,K}(X)}+1\right)^{N+1}.$$
\end{theorem}

We explain why Theorem \ref{quantitative Zhang2} reduces to Theorem \ref{quantitative Zhang3} in Proposition \ref{GenericCurve1}. 
First we reduce the general case to endomorphisms of $\bP^N$. 

\begin{proposition}\label{reduction_to_pn}
    It suffices to prove Theorem \ref{quantitative Zhang2} in the case $Y=\bP^N$, $L = O(1)$ and $\iota:\bP^N\to \bP^N$ the identity map. 
\end{proposition}

\begin{proof}
Let $(\phi,Y,L)$ be a degree $d$ polarized dynamical system as in the statement of the theorem. 
Let $\iota: Y \hookrightarrow \bP^N$ be an embedding given by a complete linear system of $L$. Since $(Y,L)$ is $d$-good, following \cite[Proposition  2.1]{Fak03} we see that there is an automorphism $M$ of $\GL_{N+1}(\ovl{K_0})$ with coordinates roots of unity (so general), and degree $d$ homogeneous map $\Phi: \bP^N\to \bP^N$ with $\Phi\circ M\circ\iota= M\circ \iota\circ \phi$. 
In particular, for all subvarieties $X'$ of $Y$ we have 
\begin{align}\label{heights equal}
\hat{h}_{\phi,L}(X')=\hat{h}_{\Phi}(M\circ\iota(X')).
\end{align}
Moreover, for each $x\in Y(\Kbar)$ we have 
\begin{align}\label{height iso}
|h(M(\iota(x)))-h(\iota(x))|\le h(M)+h(M^{-1})+h(N+1)\le 2h((N+1)!)+ 2h(N+1),
\end{align}
where in the last inequality we used that the height of the projective vector given by the entries of $M$ is $0$ and Cramer's rule. Note that expanding the determinant we have $h(\det M)\le h((N+1)!)$. 
The proposition then easily follows by \eqref{heights equal} and \eqref{height iso}.  
\end{proof}

\begin{proposition} \label{GenericCurve1}
Theorem \ref{quantitative Zhang3} implies Theorem \ref{quantitative Zhang2}.
\end{proposition}

\begin{proof}
As we have seen in Proposition \ref{reduction_to_pn} it suffices to consider endomorphisms of $\bP^N$. 
Let $K_0$, $\Phi:\bP^N\to \bP^N$ of degree $d\ge 2$ and $X\subset \bP^N$ be as in Theorem \ref{quantitative Zhang2}. Assume that Theorem \ref{quantitative Zhang3} holds. 

If $K_0$ is $\bQ$ or the function field of a curve, then by Proposition \ref{height and sup} we see that we may take $c(\Phi,\overline{K_0})$ in Theorem \ref{quantitative Zhang3} to be $C(\Phi,O(1),\mathrm{id})$. Choosing $K_1=\overline{K_0}$, Theorem \ref{quantitative Zhang2} follows in this case. 

Suppose now that $K_0$ has transcendence degree at least $2$ over $k$. We will use the construction of the generic curve from \S\ref{sec: generic} and use the notations for $k',k'',B'',B'$ therein. 
Suppose that $X$ and $\Phi$ are defined over the function field $k(B)$ or a regular projective variety over $k$ with dimension at least $2$ ($k(B)$ is a finite extension of $K_0$.) 
We will apply Proposition \ref{quantitative Zhang3} to 
$(\Phi \times_{k(B)} k''(B''))$, the subvariety $X \times_{k(B)} k'(B') \subset \bb{P}^N_{k''(B'')}$ and  $K_1 = \ovl{k(B)}\subset \overline{k''(B'')}$, noting that $B''$ is a curve over $k''$. 
Recall that from the height preservation of Proposition \ref{prop: generic}, since all the iterates of $X$ are defined over $k(B)$, we have
\begin{align}
\begin{split}
\hat{h}_{\Phi\times_{k(B)} k''(B''), k''(B'')}(\Phi^m(X)\times_{k(B)}k'(B')) &= \hat{h}_{\Phi, k(B)}(\Phi^m(X)),\\
\hat{h}_{k''(B'')}(\Phi^m(X)\times_{k(B)}k'(B'))&=\hat{h}_{k(B)}(\Phi^m(X)),~\text{ for all }m\in\bN\\
\hat{h}_{\Phi\times_{k(B)} k''(B''), k''(B'')}(x) &= \hat{h}_{\Phi, k(B)}(x), ~\text{ for all } x\in \bP^N(\overline{k(B)}).
\end{split}
\end{align}
In particular, by Proposition \ref{height and sup}, we may take $$c(\Phi,\overline{k(B)})=c(\Phi,\iota)=\sup_{x\in \bP^N(\overline{k(B)})}|\hat{h}_{\Phi,k(B)}(x)-h_{k(B)}(x)|.$$
We infer that there is a strict subvariety $Z \subsetneq X\times_{k(B)}k'(B')$, defined over $\ovl{k''(B'')}$, such that 
\begin{align}
\{x\in (X\times_{k(B)}k'(B'))(\ovl{k(B)})~:~\hat{h}_{\Phi,k(B)}(x)\leq \hat{h}_{\Phi,k(B)}(X)/2(4e)^{\dim X + 1}\}\subset Z,
\end{align}
and such that $\deg Z$ is bounded by an explicit constant given in Theorem \ref{quantitative Zhang2}. 

The only remaining issue is that $Z$ is defined
a priori over $\overline{k''(B'')}$ rather than over $\overline{k(B)}$. We now
show how to construct from $Z$ a proper subvariety of $X$ over $\overline{k(B)}$ that still
contains all points of $X(\overline{k(B)})$ with small $\Phi$-canonical height
and has the same degree as $Z$. 

Choose a finitely generated $\ovl{k(B)}$-algebra $R \subseteq \ovl{k''(B'')}$
containing all the coefficients of a fixed finite set of homogeneous polynomials
cutting out $Z$ in $\mathbb{P}^N_{\overline{k''(B'')}}$. 
Then $Z$ is defined over $\mathrm{Frac}{R}$, and there exists a closed subscheme
$\mathcal{Z} \subseteq \mathbb{P}^N_R$
whose generic fiber over $\Spec R$ is equal to $Z$. 
By generic flatness, after shrinking $\Spec R$ if necessary there is a nonempty
Zariski open subset $U \subseteq \Spec R$ such that the morphism
$\mathcal{Z}_U \to U$ is flat. 
For any point
$u : \Spec \overline{k(B)} \rightarrow U$
(i.e.\ for any $u \in U(\overline{k(B)})$), we denote by
$\mathcal{Z}_u := \mathcal{Z} \times_{\Spec R,u} \Spec \overline{k(B)}
\subseteq \mathbb{P}^N_{\overline{k(B)}}$
the corresponding fiber. Flatness implies that 
\[
\deg \mathcal{Z}_u = \deg Z \quad\text{for all }u\in U(\overline{K}).
\]
Moreover, since $Z$ is a strict subvariety of $X$, we have
$\dim Z < \dim X$, and flatness yields
\[
\dim \mathcal{Z}_u = \dim Z < \dim X
\]
for all $u\in U(\overline{K})$. Thus each $\mathcal{Z}_u$ is a proper closed
subvariety of $X_{\overline{k(B)}}$. 

It remains to show that $\cal{Z}_u(\Kbar)$ contains all points of
$X(\overline{K})$ with small $\Phi$--canonical height. Let $x \in X(\overline{k(B)})$ be such a point, so that by construction of $Z$ we have that it 
satisfies all homogeneous equations defining $Z$. 
Let $f_1,\ldots,f_r \in R[X_0,\dots,X_N]$ be homogeneous polynomials defining
$\mathcal{Z}$ in $\mathbb{P}^N_R$. Their images in
$\mathrm{Frac}(R)[X_0,\dots,X_N]$ define $Z$ on the generic fiber, so that
\[
f_j(x) = 0 \quad\text{in }\overline{k''(B'')} \qquad\text{, for all }j=1,\dots,r.
\]
These equalities hold in the fraction field $\mathrm{Frac}(R)$, and since the
coefficients of the $f_j$ lie in $R$, they already hold on $R$.
Composing the structure map $R \to \overline{k''(B'')}$ with the given
homomorphism $u^\#: R \to \overline{k(B)}$ corresponding to
$u\in U(\overline{k(B)})$ and specializing the coefficients of the $f_j$ along
$u^\#$ gives homogeneous polynomials $f_{j,u}$ over $\overline{k(B)}$ defining
$\mathcal{Z}_u$. The vanishing $f_j(x)=0$ in $\mathrm{Frac}(R)$ implies
$f_{j,u}(x)=0$ in $\overline{k(B)}$ for all $j$. Thus $x$ lies in
$\mathcal{Z}_u(\overline{k(B)})$. 
This finishes the proof.
\end{proof}

\medskip 
\noindent\emph{Proof of Theorems \ref{quantitative Zhang2}, \ref{quantitative Zhang3} and Theorem \ref{quantitative Zhang}: }
First notice that Theorem \ref{quantitative Zhang} follows immediately from Theorem \ref{quantitative Zhang2} in view of Proposition \ref{height diff} which bounds $c(\Phi,\iota)\le c_1(N,d) h(\Phi)+c_2(N,d)$. Moreover, we have seen in Proposition \ref{GublerGeneric1} that to prove Theorem \ref{quantitative Zhang2} it suffices to prove Theorem \ref{quantitative Zhang3}. Thus in what follows we work in the setting of Theorem \ref{quantitative Zhang3}. 

Let $K$ be either $\bQ$ or the function field of a regular projective curve over $k$. 
We fix an endomorphism $\Phi:\mathbb{P}_{\overline{K}}^N\to \mathbb{P}_{\overline{K}}^N$ given by degree $d\ge2$ homogeneous polynomials, as well as a subvariety $X\subset \mathbb{P}_{\overline{K}}$. 
We let $K_1$ be a subfield of $\Kbar$ and let $c(\Phi,K_1)$ denote a constant as in the statement of the Theorem. Throughout the heights are taken over $K$ and we omit it from the notation to simplify the presentation.

Let $I_X\subset \overline{K}[x_0,\ldots,x_n]$ be the homogeneous ideal defining $X$; so that $X=Z(I_X)$. Similarly, for each $m\in\bN$ we let $I_{\Phi^m(X)}$ be the defining homogeneous ideal for $\Phi^m(X)$. 
We wish to apply Proposition \ref{small section} to an appropriate iterate of $X$. 
To this end, we first we record the following lemma. 

\begin{lemma}\label{nice iterate}
For $m\in\bN$ the ideal $I_{\Phi^m(X)}$ is $D_m$-nice for 
$$D_m=\left[ (N-\dim X) \frac{\deg(X)-1}{d^m}   \right]+\dim X+1.$$ 
\end{lemma}

\begin{proof}
We follow the proof of \cite[Lemma 5.1]{DP99}. The variety $X$ is an isolated component of the complete intersection $Z$ of $N-\dim X$ hypersurfaces with degree $\le d(X)$. Therefore, $\Phi^m(X)$ is also an isolated component of $\Phi^m(Z)$. By \cite[Proposition 4.2 (i)]{DP99} we see that $Z$ is `supersympa' of degree $(N-\dim X)(d(X)-1)$. Since $\Phi^m$ has degree $d^m$, applying \cite[Proposition 4.5]{DP99} to the induced morphism of algebras $\Phi^m: \ovl{K_0}[x_0,\ldots,x_N] \to \ovl{K_0}[x_0,\ldots,x_N]$, 
we infer that the ideal associated to $\Phi^m(Z)$ is nice of degree $\left[ (N-\dim X) \frac{\deg(X)-1}{d^m}\right]+\dim X+1$ as desired.  By \cite[Proposition 4.2 (iii)]{DP99}, so is $\Phi^m(X)$. 
\end{proof}
We will also need the following lemma. 
Recall the definition of $c_0(N)$ from \eqref{nf heights compare}.

\begin{lemma}\label{choice of m}
Let $\delta>0$. Let
$$m\coloneqq \left\lceil \log_d\left(\frac{(2+6h(N+1))(\dim X+1)^2(4e)^{\dim X+1}(c(\Phi,K_1)+c_0(N)}{\hat{h}_{\phi}(X)}\right)\right\rceil.$$
For every $\epsilon<\frac{\delta (c(\Phi)+c_0)}{2}$, we have that 
\begin{align*}
\hat{h}_{\mathrm{Ph}}(\Phi^m(X))\ge (\dim X+1)(4e)^{\dim X+1}\left(\frac{2\epsilon}{\delta}+6h(N+1)\right).
\end{align*}
\end{lemma}

\begin{proof}
Recall that $c(\Phi,K_1)$ is a constant satisfying 
$$\hat{h}_{\Phi}(\Phi^m(X))\;\le\;(\dim X+1)\left(\hat{h}(\Phi^m(X))
 + c(\Phi,K_1) \right),$$
 and that by the limit Definition/Proposition \ref{canonical height def/prop} we have $\hat{h}_{\Phi}(\Phi^m(X))=d^m\hat{h}_{\Phi}(X)$. 
 Then, using also \eqref{nf heights compare}, the choice of $m$ and the bound for $\epsilon$, we obtain 
\begin{align*}
\hat{h}_{\mathrm{Ph}}(\Phi^m(X))&\ge\frac{1}{\dim X+1} \hat{h}_{\Phi}(\Phi^m(X))-c(\Phi,K_1)-c_0
=\frac{d^m}{\dim X+1}\hat{h}_{\Phi}(X)-c(\Phi,K_1)-c_0\\
&\ge (\dim X+1)(2+6h(N+1))(4e)^{\dim X+1}(c(\Phi,K_1)+c_0)- c(\Phi,K_1)-c_0\\
&\ge (\dim X+1)(1+6h(N+1))(4e)^{\dim X+1}(c(\Phi,K_1)+c_0)\\
&>(\dim X+1)\left(\frac{2\epsilon}{\delta}+6h(N+1)\right)(4e)^{\dim X+1},
\end{align*}
as claimed.
\end{proof}
We are now ready to complete the proof of Theorem \ref{quantitative Zhang3}. 
Let $m$ be as in Lemma \ref{choice of m} and note that by Lemma \ref{nice iterate} we have that $\Phi^m(X)$ is $D_m$-nice. 
Let
$$\delta_m\coloneqq  2D_m+\dim X+2,$$
and note that it is (crudely) bounded by
\begin{align}\label{deltam_bound}
\begin{split}
 \delta_m \leq 3N (\deg X)(\dim X +1).
\end{split}
\end{align}
In view of Lemma \ref{choice of m} and by Theorem \ref{David--Philippon--mult}, we see that there is a subvariety $Z\subset \Phi^m(X)$ (not necessarily irreducible) with $\deg Z\le \delta_m \deg(\Phi^m(X))$ and such that 
\begin{align}\label{weil height small}
\left\{x\in \Phi^m(X)(\overline{K}):h_{\mathrm{Ph}}(x)\le \frac{1}{(4e)^{\dim X+1}}\hat{h}_{\mathrm{Ph}}(\Phi^m(X))\right\}\subset Z.
\end{align}
Here the subvariety $Z$ comes as an intersection of the homogeneous form $q$ from Theorem \ref{David--Philippon--mult} and $\Phi^m(X)$ and we have used Bezout's theorem to infer $\deg (Z)\le \delta_m \deg(\Phi^m(X))$. 
Now recall that we have fixed a subfield $K_1 \subseteq \ovl{K_0}$ for which 
\begin{equation} \label{eq: Compare1}
\sup_{x \in \Phi^m(X)(K_1)} |\hat{h}_{\Phi}(x) - \hat{h}(x)| \leq c(\Phi,K_1).
\end{equation}
Thus we have
\begin{align*}\label{trapiterate}
\begin{split}
&\left\{x\in \Phi^m(X)(K_1):\hat{h}_{\Phi}(x)\le \frac{d^m}{2(\dim X+1)(4e)^{\dim X+1}}\hat{h}_{\Phi}(X)\right\}\subset \\
&\left\{x\in \Phi^m(X)(K_1):h_{\mathrm{Ph}}(x)\le\frac{d^m}{2(\dim X+1)(4e)^{\dim X+1}}\hat{h}_{\Phi}(X)+c(\Phi,K_1)+c_0\right\}\subset \\
&\left\{x\in \Phi^m(X)(K_1):h_{\mathrm{Ph}}(x)\le \frac{1}{(4e)^{\dim X+1}}\hat{h}_{\mathrm{Ph}}(\Phi^m(X))\right\}\subset Z,
\end{split}
\end{align*}
where in the first inclusion we used \eqref{nf heights compare}, and in the second inclusion we used \eqref{eq: Compare1} and that by our choice of $m$ we have 
\begin{align*}
 \hat{h}_{\mathrm{Ph}}(\Phi^m(X))\ge \frac{d^m}{2(\dim X+1)} \hat{h}_{\Phi}(X)+(4e)^{\dim X+1}(c(\Phi,K_1)+c_0).
\end{align*}
Now note that since $\Phi$ is defined over $K_1$ and canonical heights satisfy $\hat{h}_{\Phi}(\Phi^m(z))=d^m\hat{h}_{\Phi}(z)$ 
we may infer 
\begin{equation}\label{trapiterate2}
\begin{split}
&\left\{z\in X(K_1):\hat{h}_{\Phi}(z)\le \frac{1}{2(\dim X+1)(4e)^{\dim X+1}}\hat{h}_{\Phi}(X)\right\}\subset \Phi^{-m}(Z).
\end{split}
\end{equation}

Moreover, 
\begin{align*}
&\deg \phi^{-m}(Z)= d^{m (N - \dim Z)} \deg Z\le d^{m (N- \dim Z)} \delta_m\deg(\Phi^m(X))\le\\
&\le d^{m (N-\dim Z)} \delta_m d^{m\dim \Phi^m(X)}\deg(X)\\
& \le 3N (\dim X+1) (\deg X)^2\left(\frac{(2+6h(N+1))(\dim X+1)^2(4e)^{\dim X+1}(c(\Phi,K_1)+c_0)}{\hat{h}_{\phi}(X)}+1\right)^{N+1},
\end{align*}
where we used the fact that $Z=\Phi^m(X)\cap Z(q)$ and the bound for $\delta_m$ from \eqref{deltam_bound}. 
This completes the proof of Theorem \ref{quantitative Zhang3}. 

 \qed

\section{Applications towards effectivity in the dynamical Bogomolov Conjecture} \label{sec: DynamicalBogomolov}

In this section our goal is to prove Theorem \ref{oberlyff}. Given our quantitative version of the fundamental inequality, it is easy to see that it suffices to establish the following result.

\begin{theorem}\label{heightboundfermat}
Let $K_0$ be either $\bQ$ or a function field. 
Let $d\ge 2$ and let $C_n:~x^n+y^n=z^n$ be the Fermat curve in $\mathbb{P}^2$ for $n\in\bN$.
Then for every regular polynomial endomorphism $\Phi:\mathbb{P}_{\overline{K_0}}^2\to \mathbb{P}_{\overline{K_0}}^2$ such that $\Phi|_{L_{\infty}}\in\overline{K_0}(z)$ has degree $\ge 2$,  the following hold. 
\begin{enumerate}
\item If $K_0$ is a function field, $n\ge 2^{21} d^6$, $\mathrm{char}(K_0)\nmid n$, and $\Phi|_{L_{\infty}}$ is not defined over the field of constants of $K_0$, then 
$$\hat{h}_{\Phi}(C_n)\ge \frac{1}{8(4e)^2(\sqrt{d}+1)^2}  h(\Phi|_{L_{\infty}}).$$
\item If $K_0=\bQ$, $n\ge 2^{1600d^3}$, and $\Phi|_{L_{\infty}}$ is a monic polynomial that is not a power map, then $$\hat{h}_{\Phi}(C_n)\ge \frac{1}{8 (4e)^{2}2^{500d^3}}\max\{h(\Phi|_{L_{\infty}}),1\}.$$ 
\end{enumerate}
\end{theorem}


To avoid repetition, we fix a polynomial endomorphism $\Phi:\mathbb{P}^2_{\overline{K_0}}\to\mathbb{P}^2_{\overline{K_0}}$, an integer $n\ge 1$, and the Fermat curve $C_n:x^n+y^n=z^n\subset\mathbb{P}^2$. Let $f\coloneqq \Phi|_{L_\infty}$, identifying $L_\infty\simeq\mathbb{P}^1$, so that $f:\mathbb{P}^1_{\overline{K_0}}\to\mathbb{P}^1_{\overline{K_0}}$ is a rational map of degree $\ell\ge 2$. We also consider the product map $(f,\phi_\ell):\mathbb{P}^1\times\mathbb{P}^1\to\mathbb{P}^1\times\mathbb{P}^1$ given by $(x,y)\mapsto(f(x),y^\ell)$. Finally, we fix the line bundle $O(1,1)=\pi_1^{*}O(1)\otimes\pi_2^{*}O(1)$ on $\mathbb{P}^1\times\mathbb{P}^1$, where $\pi_i$ denotes the projection to the $i$-th factor.

If additional assumptions on $K_0$, $\Phi$, $f$, or $n$ are required, we will state them explicitly in the relevant lemmas, so that it is clear at which point each hypothesis becomes essential.

Our first step is to relate the $\Phi$-canonical height of the Fermat curve to the $(f,\phi_{\ell})$-canonical height of the diagonal in $\bP^1\times\bP^1$. 

\begin{proposition} \label{QuantitativeBound1}
Let $\ell\ge 2$ and $c_1(1,\ell),c_2(1,\ell)$ as in Proposition \ref{height diff}. 
Assume that 
\begin{align*}
n\ge 144\left(
     (2+6h(4))(4e)^{2}4^2
     \frac{\bigl(c_1(1,\ell)h(f)+c_2(1,\ell)+\tfrac{7}{2}h(4)\bigr)}
          {\hat{h}_{(f,\phi_{\ell}),O(1,1)}(\Delta)} +1 \right)^3,
\end{align*}
and $\mathrm{char}(K_0)\nmid n$. Then 
\begin{align*}
\hat{h}_{\Phi}(C_n)\ge  \frac{1}{8 (4e)^{2}}\hat{h}_{(f,\phi_{\ell}),O(1,1)}(\Delta).
\end{align*}
\end{proposition}

\begin{proof}

By Definition/Proposition \ref{height def} (3), the definition of the intersection pairing in \cite{Gub98, CLT09}, the construction of the $\Phi$-invariant metric, and the assumption that $\Phi$ is a regular polynomial endomorphism, we infer that 
\begin{equation}\label{fermattodiagonal}
\hat{h}_{\Phi}(C_n)= \frac{\overline{L}^2_{\Phi}\cdot C_n}{\deg C_n} \ge \frac{1}{n} \left(\displaystyle\sum_{x^n=-1 }\hat{h}_f(x) \right),
\end{equation}
exactly as in \cite[equation (4)]{DFR23}. 

Next we wish to apply Theorem \ref{quantitative Zhang2} to triple $((f,\phi_{\ell}), \bP^1\times\bP^1, O(1,1))$ and to the diagonal subvariety $\Delta\subset \bP^1\times \bP^1$. 
To this end, we fix the Segre embedding $s: \bP^1\times\bP^1\hookrightarrow \bP^3$ and note that $h(s(x,y))=h(x)+h(y)$ for $x,y\in \bP^1(\overline{K_0})$ so that by Proposition \ref{height diff} we have 
$$C((f,\phi_{\ell}),s)=\sup_{x\in \bP^1(\overline{K_0})}|\hat{h}_f(x)-h(x)|\le c_1(1,\ell)h(f)+c_2(1,\ell).$$ 
Note also that by the projection formula $\deg_{O(1,1)}\Delta= 2$. 
We infer that the number of roots of unity $\xi$ such that 
$\hat{h}_f(\xi)<\frac{1}{4 (4e)^{2}}\hat{h}_{(f,\phi_{\ell}),O(1,1)}(\Delta)$ is at most 
\begin{align*}
72\left(
     \frac{(2+6h(4))(4e)^{2}4^2
     \bigl(c_1(1,\ell)h(f)+c_2(1,\ell)+\tfrac{7}{2}h(4)\bigr)}
          {\hat{h}_{(f,\phi_{\ell}),O(1,1)}(\Delta)} +1 \right)^3.
\end{align*}
The conclusion then follows from a simple counting by our choice of $n$ and \eqref{fermattodiagonal}.
\end{proof}

\begin{remark}
Similar consequences could have been inferred from \cite[Theorem B]{Gau24} upon a more precise analysis of the constants appearing in the author's proof. 
\end{remark}

The next two propositions allow us to compare the $(f,\phi_{\ell})$-height of the diagonal  with the height of $f$. 
The first one is a simple consequence of results from \cite{Obe24}.
\begin{proposition}\label{fewrootsofunity}
We have 
\begin{align*}
\hat{h}_{(f,\phi_{\ell}),O(1,1)}(\Delta)\ge \frac{1}{(\sqrt{\ell}+1)^2} h_{\mathrm{Ph}}(f)-\frac{3(\ell+2)}{2(\sqrt{\ell}+1)^2} h(2)-h(\sqrt{2}).
\end{align*}
\end{proposition}

\begin{proof}
This follows from the work in \cite{Obe24}. 
We will explain how using the notations in loc.cit.
First note that the Arakelov--Zhang pairing of maps used in \cite{Obe24} and studied in \cite{PST12, FRL06, Fil17}, can be directly related with the canonical height of the diagonal 
\begin{align}\label{canheight=az}
(f,\phi_{\ell})_{\mathrm{AZ}}=\hat{h}_{(f,\phi_{\ell}),O(1,1)}(\Delta);
\end{align}
see \cite[Lemma 2.3]{Gau24}. 
Moreover, as in \cite[Proposition 5.12]{Obe24}, for any sequence $\{a_n\}\subset \overline{K_0}$ with $\hat{h}_f(a_n)\to 0$ as $n\to \infty$ we have that 
$h_{\mathrm{Ph}}(a_n)\to \frac{1}{2}\|\mu_f\|^2$ and $h(a_n)\to (f,\phi_{\ell})_{\mathrm{AZ}}$. 
and noting that $0\le h_{\mathrm{Ph}}-h\le h(\sqrt{2})$, we infer that 
\begin{equation}\label{easy}
(f,\phi_{\ell})_{\mathrm{AZ}}\ge \frac{1}{2}\|\mu_f\|^2- h(\sqrt{2}).
\end{equation}
In the case that $K_0=\bQ$ the result then follows from \cite[equation 5.9]{Obe24} since
\begin{equation} \label{squarevsazheight}
\begin{aligned}
\begin{split}
\frac{1}{2}\|\mu_f\|^2&\ge \frac{1}{(\sqrt{\ell}+1)^2} h_{\mathrm{Ph}}(f)- \frac{1}{(\sqrt{\ell}+1)^2} \left( \frac{3}{2}(\ell+1)\log2 +\frac{1}{4}\right)\\ 
&\ge \frac{1}{(\sqrt{\ell}+1)^2} h_{\mathrm{Ph}}(f)- \frac{3(\ell+2)}{2(\sqrt{\ell}+1)^2} \log2.
\end{split}
\end{aligned}
\end{equation}
Though the author in \cite{Obe24} formulates their results over a number field, the proofs extend verbatim to any product formula field. For the potential theoretic background in this generality we refer the reader to \cite{BR10}. Here, we explain briefly how the constants  improve in the absence of archimedean places. 
Suppose that $K_0$ is a function field. 
We let $\lambda_{\mathrm{Ar}}$ denote the adelic measure given by the delta mass at the Gauss point at each place $v\in M_K$ (constructed from the places of $M_{K_0}$) and note that in the absence of archimedean places we have that the corresponding $t_v=0$ at each place (see the discussion in \cite[page 8, proof of theorem 3.1]{Obe24}.)
First note that by definition (compare \cite[equation 3.6]{Obe24}) we have 
\begin{align}\label{az norm}
2(f,\phi_{\ell})_{\mathrm{AZ}}= \|\mu_f\|^2_{\lambda_{\mathrm{Ar}}}\coloneqq \|\mu_f\|^2.
\end{align}
Summing up \cite[Lemma 5.6]{Obe24} over all the places of $K$ we also have 
\begin{align}\label{pullbackvsheight}
\frac{1}{2\ell}{\|f^{*}\lambda_{\mathrm{Ar}}\|^2}= h_{\mathrm{Ph}}(f).
\end{align}
Note that from \cite[Proposition 3.2]{Obe24} (since $t=0$) we have 
\begin{align*}
\|\lambda_{\mathrm{Ar}}-\mu_f\|^2=\|\mu_f\|^2,
\end{align*}
and using also the fact that $\ell^{-1}f^*\mu_f=\mu_f$, we get
\begin{align*}\label{totriange}
\begin{split}
\|\ell^{-1}f^*(\lambda_{\mathrm{Ar}})-\mu_f\|^2&=\|\ell^{-1}f^*(\lambda_{\mathrm{Ar}}-\mu_f)\|=\ell^{-1}\|\lambda_{\mathrm{Ar}}-\mu_f\|^2=\ell^{-1}\|\mu_f\|^2.
\end{split}
\end{align*}
Now by the triangular inequality for $\|\cdot\|$ we have 
\begin{align*}
\left| \|\ell^{-1}f^*(\lambda_{\mathrm{Ar}})\|-\|\mu_f\|\right|\le \|\ell^{-1}f^*(\lambda_{\mathrm{Ar}})-\mu_f\|=\ell^{-1/2}\|\mu_f\|,
\end{align*}
so that recalling \eqref{az norm} and \eqref{pullbackvsheight} we get
\begin{align*}
2(f,\phi_{\ell})_{\mathrm{AZ}}=\|\mu_f\|^2\ge \left(\frac{1}{1+\ell^{-1/2}}\right)^2\|\ell^{-1}f^*(\lambda_{\mathrm{Ar}})\|^2= \left(\frac{1}{1+\ell^{-1/2}}\right)^2\frac{2}{\ell}h_{\mathrm{Ph}}(f).
\end{align*}
The Proposition follows. 
\end{proof}
In the number field case we will also need a uniform lower bound for the canonical height of the diagonal. This is the main reason for restricting Theorem~\ref{oberlyff} to polynomial maps on the line at infinity. Our approach follows \cite{AY23}. We begin with the following simple lemma. 

\begin{lemma}\label{BumpFunctionEstimate}
There exists a smooth function $\psi_0:\mathbb{C}\to\mathbb{R}$ such that $\psi_0=1$ on $B(0,1)$, $\psi_0=0$ outside $B(0,2)$, and $\|\nabla\psi_0\|\le 2$.
\end{lemma}

\begin{proof}
Let $\rho(t)=e^{-1/t}$ for $t>0$ and $\rho(t)=0$ for $t\le 0$, and define
\[
\eta(t)=\frac{\rho(t)}{\rho(t)+\rho(1-t)}.
\]
Set $\psi_0(x)=\eta(2-\|x\|).$
It is clear from the definition that $\psi_0=1$ on $B(0,1)$ and $\psi_0=0$ outside $B(0,2)$.
Moreover, the chain rule gives
\[
\nabla\psi_0(x)=\eta'(2-\|x\|)\,\nabla(2-\|x\|)
=-\eta'(2-\|x\|)\,\frac{x}{\|x\|}.
\]  
A direct computation shows that $|\eta'|\le 2$ on $\mathbb{R}$, and therefore $\|\nabla\psi_0\|\le 2$, as required.
\end{proof}

The lower bound on the height of the diagonal is as follows. 

\begin{proposition} \label{LowerBoundPairingNumberField}
Let $f\in\Qbar[z]$ be a monic polynomial with $\deg f=\ell\ge 2$ that is not a power map. We have 
\begin{align}
\hat{h}_{(f,\phi_{\ell}),O(1,1)}(\Delta)\ge 2^{-500\ell^3}\max\{h(f),1\}\ge 2^{-510\ell^3}(c_1(1,\ell)h(f)+ c_2(1,\ell)+7/2\log4).
\end{align}
\end{proposition}

\begin{proof}
By Proposition \ref{fewrootsofunity}, it is easy to see that if $h_{\Ph}(f) \geq 30(\ell+1)\log 2$, we have 
\begin{align*}
\hat{h}_{(f,\phi_{\ell}),O(1,1)}(\Delta)\ge \frac{1}{6\ell}h_{\mathrm{Ph}}(f)\ge \frac{1}{6\ell}h(f).
\end{align*}
For this we used the fact that $3\ell\geq (\sqrt{\ell}+1)^2$.
It remains to handle the case of $f$ with small height, and we henceforth let 
$f(z)=z^{\ell}+a_{1}z^{\ell-1}+\cdots+a_{\ell}\in K[z]$
be a monic polynomial defined over a finite extension $K$ of $\bQ$ with 
 $h(f) \leq 30(\ell+1) \log 2$. 
We write $C_{\ell}\coloneqq  2^{30(\ell+1)2\ell}$.  
Let $i\in \{1,\ldots,\ell\}$. Then
$$\sum_{v\in M_K^{\infty}}\frac{[K_v:\bQ_v]}{[K:\bQ]}\log^{+}|a_i|_v\le h(a_i)\le h(f)\le 30(\ell+1) \log 2,$$ 
so that letting $N_v\coloneqq  [K_v:\bQ_v]/[K:\bQ]$ we have 
\begin{align}\label{bounded height 1}
\sum_{v\in M_K^{\infty}~:~|a_i|_v \le C_{\ell}}N_v \ge \frac{2\ell-1}{2\ell},
\end{align}
and since $h(a_i)=h(a_i^{-1})$, we similarly have 
\begin{align}\label{bounded height 2}
\sum_{v\in M_K^{\infty}~:~|a_i|^{-1}_v \le C_{\ell}}N_v\ge \frac{2\ell-1}{2\ell}.
\end{align}
From \eqref{bounded height 1} and \eqref{bounded height 2} we infer that if 
$S=\{v\in M_K^{\infty}:\frac{1}{C_{\ell}}\le |a_i|_v \le C_{\ell}\}$, then 
\begin{align}\label{many good places}
\displaystyle\sum_{v\in S}N_v \ge1 - 1/\ell\ge 1/2.
\end{align}
Recall from \eqref{canheight=az} that the canonical height agrees with the Arakelov--Zhang pairing, which can in turn be expressed as the sum of local energy pairings coming from each place of $K$. More precisely from \cite{FRL06} and \cite{PST12}, we have
\begin{align}
\hat{h}_{(f,\phi_{\ell}),O(1,1)}(\Delta)=(f,\phi_{\ell})_{\mathrm{AZ}}= \sum_{v\in M_K} N_v (\mu_{f,v} - \mu_{\phi_{\ell},v}, \mu_{f,v} - \mu_{\phi_{\ell},v})_v,
\end{align}
where for each $v \in M_K$,  $\mu_{f,v}$ is the equilibrium measure of $f$ on $\berkP$ \cite[Section 6.1]{FRL06}, 
$\mu_{\phi_{\ell},v}$ is the uniform probability measure on the unit circle $S^1$ if $v$ is archimedean and the delta mass at the Gauss point of $\berkA_v$ if $v$ is non-archimedean, and 
$$(\mu, \mu')_v=-\iint_{\berkA \times \berkA \setminus \Diag} \log |z-w| d \mu(z) d \mu'(w)$$
for any two signed measures on $\berkP$. 
Note that for each $v$ the local pairing is non-negative \cite[Proposition 2.6, Proposition 4.5]{FRL06} so that 
\begin{align}\label{lower bound over S}
\hat{h}_{(f,\phi_{\ell}),O(1,1)}(\Delta)\ge \sum_{v\in S} N_v (\mu_{f,v} - \mu_{\phi_{\ell},v} , \mu_{f,v} - \mu_{\phi_{\ell},v})_v,
\end{align}

Now let $v\in S$. 
We will follow the strategy in \cite[\S 3.4]{AY23} to get a lower bound for $(\mu_{f,v}, \mu_{S^1,v})_v$. 
First note that by \cite[Lemma 3.21]{AY23}, and since $v\in S$, if  $j\coloneqq j(f)\in\{1,\ldots,\ell\}$ is the smallest integer such that $a_j\neq 0$, which exists since $f$ is not a power map, then 
\begin{align}\label{lower difference}
\left|\int z^{j} d(\mu_{f,v} - \mu_{S^1,v}) \right| = \frac{j}{\ell} \left|a_{j}\right|_v \geq \frac{1}{\ell C_{\ell}}.
\end{align}
Moreover, by \cite[Lemma 3.19]{AY23} we know that for any smooth function $\psi: \bC\to \bR$ we have 
\begin{align}\label{pairing nabla}
 \left|\int \psi  d(\mu_{f,v} - \mu_{S^1,v}) \right|^2 \leq \|\nabla \psi\|^2_{L^2} (\mu_{f,v} - \mu_{\phi_{\ell},v}, \mu_{f,v} - \mu_{\phi_{\ell},v})_v.
\end{align}
We will now construct an appropriate function $\psi$ that will allow us to infer a lower bound of the energy pairing from \eqref{lower difference} and \eqref{pairing nabla}. 
First let $\psi_0$ be such a smooth function with $\|\nabla \psi_0\|\le 2$ as in Lemma \ref{BumpFunctionEstimate}. 
Then taking $\psi_1 = z^j \psi_0$, we obtain a smooth function that is $z^j$ on $B(0,1)$, $0$ outside of $B(0,2)$ and 
\begin{align*}
\|\nabla \psi_1\|_{L^2}\le 2\sqrt{\pi}\sup_{|z|\le 2} \|\nabla \psi_1\| &=2\sqrt{\pi}\sup_{|z|\le 2} \|z^j \nabla \psi_0 + \psi_0 \nabla z^j\|\le 2\sqrt{\pi}(2^{j+1} + \sqrt{2} j 2^{j-1}).
\end{align*}
Finally, we set $\psi = (3C_{\ell})^j \psi_1(\frac{z}{3C_{\ell}})$ so that 
\begin{align}\label{psi bound}
\|\nabla \psi\|_{L^2} = (3C_{\ell})^{j} \|\nabla \psi_1\|_{L^2} \leq 8 (3C_{\ell})^{j} j 2^{j+1}\le 8 (3C_{\ell})^{\ell} \ell 2^{\ell+1}
\end{align}
Note that $\psi=z^j$ on $B(0, 3C_{\ell})$ and in particular throughout the support of $\mu_{f,v}$ and $\mu_{S^1,v}$ by \cite[Lemma 3.2]{AY23}. Thus combining \eqref{lower difference}, \eqref{psi bound} and \eqref{pairing nabla} we infer
$$\frac{1}{({\ell} C_{\ell})^2} \leq  8^2 (3^2C^2_{\ell})^{\ell} \ell ^24^{\ell+1}(\mu_{f,v} - \mu_{\phi_{\ell},v}, \mu_{f,v} - \mu_{\phi_{\ell},v})_v .$$
Summing over all $v\in S$ and in view of \eqref{lower bound over S} and \eqref{many good places} we infer 
\begin{align}
\hat{h}_{(f,\phi_{\ell}),O(1,1)}(\Delta)\ge \frac{1}{128\cdot 9^{\ell+1}2^{2(\ell+1)}\ell^4 C^{2\ell+2}_{\ell} }\ge \frac{1}{128\ell^4 \cdot 9^{\ell+1}\cdot 2^{30(\ell+1)2\ell(2\ell+2)+2(\ell+1)}}.
\end{align}
Recall that $c_1(1,\ell) = 2 \ell + 1 , ~c_2(1, \ell) = 2 (\ell + 1) (\ell+1)^{2 (\ell+2)}$. 
The proposition then follows after a very crude estimate. 
\end{proof}

\noindent\emph{Proof of Theorem \ref{heightboundfermat}: }In the case of function fields the conclusion follows immediately from Propositions \ref{QuantitativeBound1} and \ref{fewrootsofunity}. Recall that in the case of function fields $c_2(1,\ell)=0$ and $c_1(1,\ell)=2\ell+1$, so the lower bound for $n$ from Proposition \ref{QuantitativeBound1} is satisfied if
\begin{align*}
n\ge 2\cdot 10^6 \cdot \ell^6\ge 144\left(
     \frac{2(4e)^{2}4^2
     (2\ell+1)h(f)}
          {\hat{h}_{(f,\phi_{\ell}),O(1,1)}(\Delta)} +1 \right)^3 
\end{align*}
Now assume that $K_0=\bQ$. 
The result follows from Propositions \ref{QuantitativeBound1}, and \ref{LowerBoundPairingNumberField} once we observe that in this case we can take
\begin{align*}
\begin{split}
n&\ge 2^{1600\ell^3}\ge 144\left(
     (2+6\log 4(4e)^{2}4^2
     2^{510\ell^3} +1 \right)^3\\
     &\ge 144\left(
     (2+6h(4))(4e)^{2}4^2
     \frac{\bigl(c_1(1,\ell)h(f)+c_2(1,\ell)+\tfrac{7}{2}h(4)\bigr)}
          {\hat{h}_{(f,\phi_{\ell}),O(1,1)}(\Delta)} +1 \right)^3
\end{split}
\end{align*}

\qed

\noindent\emph{Proof of Theorem \ref{oberlyff}: }
This follows from Theorem \ref{quantitative Zhang} in view of  Theorem \ref{heightboundfermat}. \qed


\section{A geometric gap principle \`a la Gao--Ge--K\"uhne} \label{sec: Bogomolov}

The goal of this section is to prove Theorem~\ref{uniform_geometric_Bogomolov} and Theorem~\ref{uniform CGHX}.  
Our strategy is to apply Theorem~\ref{quantitative Zhang2}.  
To this end, we must establish three auxiliary results: a “height specialization result,” a “geometric height gap,” and a “geometric height comparison,” proved respectively in 
\S\ref{specialization_av}, \S\ref{gap}, and \S\ref{diff_height}. 
We complete the proofs of Theorems \ref{uniform_geometric_Bogomolov} and \ref{uniform CGHX} in \S\ref{proof of uniform B}.
Throughout this section we freely use Definition/Proposition~\ref{canonical height def/prop}(3), without further mention, in order to invoke properties of the N\'eron–Tate height of subvarieties that are valid for either equivalent definition.
 
Before we begin we recall some background on heights for abelian varieties.

\subsection{Heights of abelian varieties}\label{heights on av}
Unless specified otherwise, we denote by $k$ an algebraically closed field of characteristic zero and by $K=k(B)$ the function field of a projective variety $B$ over $k$ that is normal. 
We also fix an ample line bundle $\cal{M}$ on $B$. 

For $g\ge 2$, let $\mathcal{A}_{g,3}$ denote the fine moduli space of principally polarized abelian varieties of dimension $g$ with level-$3$ structure; it is a quasi-projective variety over $\overline{\mathbb{Q}}$ (see \cite[Ch.~7]{MFK94}). A point of $\mathcal{A}_{g,3}$ corresponds to a triple $(A,\phi_L,\nu)$, where $A$ is an abelian variety of dimension $g$, $\phi_L$ is a principal polarization corresponding to an ample symmetric line bundle $L$ on $A$, and $\nu$ is a level-$3$ structure. 

By Faltings--Chai \cite{FC90}, the space $\mathcal{A}_{g,3}$ admits a minimal compactification $\mathcal{A}_{g,3}^*$ over $\overline{\mathbb{Q}}$, equipped with an ample line bundle $\omega$ (the Hodge bundle). Let $\bm{\omega}\coloneqq \omega^{\otimes M}$ for $M\gg 0$, chosen so that $\bm{\omega}$ is very ample. This yields an embedding
\[
j_{\bm{\omega}}:\mathcal{A}_{g,3}^*\hookrightarrow \mathbb{P}^m
\]
for some $m$, over $\overline{\mathbb{Q}}$. We fix a representative for the Weil height corresponding to $\bm{\omega}$ by
$$h_{\bm\omega}(\cdot)\coloneqq  \frac{1}{M} h_{\bP^m}(j_{\bm{\omega}}(\cdot)).$$  
For the function field $K$, we write $$h_{\bm{\omega}, K}(\cdot)\coloneqq  \frac{1}{M} h_{\bP^m,K}(j_{\bm{\omega}}(\cdot)).$$

Note that for function fields $K$ of transcendence degree one, this height coincides with the stable Faltings height (or differential height) of the underlying abelian variety; see \cite{GLFP25,Fal83,MB85}.  More precisely, for any $(A,\phi_L,\nu)\in\mathcal{A}_{g,3}(\overline{K})$ we have
\begin{align}\label{faltings=modular}
h_{\mathrm{Fal}}(A)
= h_{\boldsymbol{\omega}}\bigl((A,\phi_L,\nu)\bigr)
\coloneqq h_{\boldsymbol{\omega}}(A),
\end{align}
so in particular the height is independent of the choice of principal polarization and level structure. 

Recall that the stable Faltings height is defined as follows.  Let $N/B$ denote the Néron model of $A/K$, and $e:B\to N$ the identity section. Then
$h_{\mathrm{Fal}}(A)
= \deg_B\!\bigl(e^{*}(\wedge^{g}\Omega^{1}_{N/B})\bigr).$
One must replace $K$ by a finite extension over which $A$ acquires semistable reduction; the resulting quantity is independent of the choice of such an extension, hence the name \emph{stable} Faltings height.
Note also that the independence of $h_{\bm\omega}$ of $L,\nu$ remains true for higher dimensional function fields, using the generic curve as in Proposition \ref{prop: generic}.

For a general abelian variety $A$, not necessarily principally polarized, we define 
\[
h_{\bm\omega}(A)\;:=\;h_{\bm\omega}(A_0),
\]
where $A_0$ is any principally polarized abelian variety that is isogenous to $A$.  
This is well defined by \cite[Theorem~A]{GLFP25}, which shows that the height is independent of the choice of such an $A_0$.

\subsection{A height specialization result}\label{specialization_av}

The next tool we require is a generalization of Silverman's specialization theorem \cite{Sil83} to higher-dimensional subvarieties. Such a generalization has been established by Ingram \cite[Theorem~15]{Ing22} for hypersurfaces and by Gauthier--Vigny \cite{GV24} over $\Qbar$. However, for our purposes we will need a version that applies in all dimensions and over arbitrary function fields as well.

\begin{definition}\label{model}
Let $C$ be a smooth projective curve over an algebraically closed field $k$. 
Let $A$ be an abelian variety, $L$ a symmetric ample line bundle on $A$, and $X \subset A$ a closed subvariety, all defined over $k(C)$. 
A \emph{model} of the triple $(A,L,X)$ over $C$ consists of data
$
(U,\, \pi : \mathcal{A} \to U,\, \mathcal{L},\, \mathcal{X}),
$
where
\begin{itemize}
  \item $U \subset C$ is a nonempty Zariski-open subset;
  \item $\pi : \mathcal{A} \to U$ is an abelian scheme over $k$ (i.e. a smooth and proper group scheme with connected fibers) whose generic fiber is isomorphic to $A$;
  \item $\mathcal{L}$ is a line bundle on $\mathcal{A}$ that is relatively ample, symmetric, and restricts to $L$ on the generic fiber;
  \item $\mathcal{X} \subset \mathcal{A}$ is a closed subscheme, flat over $U$, whose generic fiber is $X$.
\end{itemize}
For $s \in U(k)$, we denote the fibers by
$\mathcal{A}_s, \mathcal{L}_s, \mathcal{X}_s$. 
We note that $(A,[2],L)$ and $(\cal{A}_s,[2],\mathcal{L}_s)$ are polarized dynamical systems and we write 
\begin{align*}
 \hat{h}_{A,L}\coloneqq  \hat{h}_{[2],L}   ; ~ \hat{h}_{\cal{A}_s,\cal{L}_s}\coloneqq  \hat{h}_{[2],\cal{L}_s},
\end{align*}
for the canonical heights on subvarieties of $A$ and $\cal{A}_s$ respectively. 
\end{definition}

\begin{proposition}\label{SilvermanSpecialization1}
\label{prop:Silverman-specialization}
Let $K_1$ be either $\Qbar$ or $\overline{\mathbb{Q}}(V)$, the function field of a regular positive-dimensional projective variety $V $ over $\overline{\mathbb{Q}}$. Let $C$ be a smooth projective curve over $K_1$.
Set  $K_2 = \overline{K_1}(C)$. 
Let $A$ be an abelian variety, $L$ a symmetric ample line bundle on $A$, and $X \subset A$ a closed subvariety, all defined over $K_2$. Let
\((U,\mathcal{A},\mathcal{L},\mathcal{X})\)
be a model of $(A,L,X)$ over $U \subset C$. 
Let $h_{C,K_1} : C(\overline{K_1}) \to \mathbb{R}_{\ge 0}$
be a height function associated to an ample divisor on $C$ of degree $1 $.
Then
\begin{equation*}
\lim_{\substack{s \in U(\overline{K_1}) \\ h_{C,K_1}(s) \to \infty}}
\frac{
  \widehat{h}_{\mathcal{A}_s, \mathcal{L}_s, K_1}(\mathcal{X}_s)
}{
  h_{C,K_1}(s)
}
=
\widehat{h}_{A, L}(X).
\end{equation*}
Here $\widehat{h}_{\mathcal{A}_s, \mathcal{L}_s, K_1}$ denotes the Néron–Tate height
on the fiber $\mathcal{A}_s$,
and $\widehat{h}_{A,L}(X)$ is the canonical height of $X$ over $K_2$.
\end{proposition}

\begin{proof}
 
Recall that by \cite[Corollary 2.2]{Fak03} there exists an embedding $\iota:A\hookrightarrow \bP^N$, given by a complete linear system of $L$ and a morphism $\phi: \bP^N\to \bP^N$ such that $\phi\circ \iota= \iota\circ [2]$ and note that $\hat{h}_{A,L}(X)=\hat{h}_{\phi}(\iota(X))$ for all subvarieties $X\subset A_{K_2}$. 
Let
\((U,\mathcal{A},\mathcal{L},\mathcal{X})\)
be a model of $(A,L,X)$ over $U \subset C$. 
After possibly shrinking $U$, the embedding $\iota$ and the morphism $\phi$ induce an embedding $\iota_U: \cal{A}\hookrightarrow \cal{A}$ and a morphism 
$\phi_U: U\times \bP^N\to U\times \bP^N$, with generic fibers $\iota$ and $\phi$ respectively and such that for each $t\in U(\overline{K_1})$ we have $\hat{h}_{\cal{A}_t,\cal{L}_t}(\cal{X}_t)=\hat{h}_{\phi_t}(\iota_t(\cal{X}_t)).$ 
We further have that $\deg(\phi_t)=\deg(\phi)$ for each $t$. 
Let $r\coloneqq \dim (\cal{X}_t)=\dim (\iota_t(\cal{X}_t))$ for every $t\in U(\overline{K_1})$ by flatness. 
The proof of the proposition now follows exactly as in \cite[Theorem 15]{Ing22}, which itself extends the arguments in \cite{Sil83} to the higher dimensional case. 
We repeat the key steps here for completeness. 

Let $c_1\coloneqq c_1(N,\deg(\phi))$ and $c_2\coloneqq c_2(N,\deg(\phi))$ be as in Proposition \ref{height diff} and let $c_0\coloneqq 7/2h_{K_1}(N+1)$. 
Applying Proposition \ref{height diff} , by the height comparison in Definition/Proposition \ref{canonical height def/prop}(3) and by \eqref{nf heights compare}, we have 
\begin{align}\label{ingram thing}
    |\hat{h}_{\phi_t,K_1}(\iota_t(\cal{X}_t))- \hat{h}_{\Ph,K_1}(\iota_t(\cal{X}_t))|\le r c_{1} h_{K_1}(\phi_t) + r c_2+ r c_0. 
\end{align}
Let $h_C\coloneqq h_{C,K_1}$ be a height on $C$ as in the statement. 
Note that 
\begin{align}\label{easy spec}
\lim_{h_C(t)\to\infty}\frac{\hat{h}_{\Ph,K_1}(\phi_t)}{h_C(t)}=h_{\Ph,K_2}(\phi),~\lim_{h_C(t)\to\infty}\frac{\hat{h}_{\Ph,K_1}(\iota_t(\cal{X}_t))}{h_C(t)}=\hat{h}_{\Ph, K_2}(\iota(X)),
\end{align}
which is easy to see since the Philippon height of a subvariety is given by the height of its Chow form, which is now a point in projective space. 
From \eqref{ingram thing} and \eqref{easy spec} and the triangular inequality exactly as in \cite{Ing22}, we infer  
\begin{align}\label{to be iterated}
\limsup_{h_C(t)\to\infty}\left|\frac{\hat{h}_{\phi_t}(\iota_t(\cal{X}_t))}{h_C(t)} - \hat{h}_{\phi}(\iota(X))\right|\le rc_1.
\end{align}
Applying \eqref{to be iterated} to the iterates $\phi_t^n(\iota_t(\cal{X}_t))=\iota_t([2^n]\cal{X}_t)$ for all $n$ we get 
\begin{align*}
\limsup_{h_C(t)\to\infty}\left|\frac{\hat{h}_{\phi_t}(\iota_t(\cal{X}_t))}{h_C(t)} - \hat{h}_{\phi}(\iota(X))\right|\le \frac{rc_1}{d^n}.
\end{align*}
The conclusion follows letting $n\to\infty$.
\end{proof}
\subsection{A height gap}\label{gap}
Recall that an abelian variety $A$ is generated by a subvariety $X\subset A$ if it is the smallest
abelian subvariety containing $X-X$.  
The Ueno locus of a subvariety $X\subset A$  is the union of all positive dimensional cosets contained in $X$. Its complement in $X$ is denoted $X^{\circ}$ and is Zariski open; see \cite{Kaw80}. 
We will prove the following result. 

\begin{theorem} \label{GeometricHeightGap2}
Let $g\ge 2$ and $D\ge 1$. There is a constant $c_{\mathrm{GGK}}\coloneqq c_{\mathrm{GGK}}(g,D)>0$  with the following property.  
Let $K$ a characteristic $0$ function field, let $(A,\phi_L,\nu)\in \cal{A}_{g,3}(K)$ and let $X \subseteq A_K$ be an irreducible subvariety that generates $A$ with $\deg_L X \leq D$ and $X^{\circ} \not = 0$. Then 
$$\h_{A,L,K}(X) \geq c_{\mathrm{GGK}} h_{\bm\omega,K}(A).$$
\end{theorem}

First we record the following consequence of Dimitrov--Gao--Habegger's height inequality \cite{DGH21}, as interpreted by Gao--Ge--K\"uhne's \cite[Proposition 4.1, 4.2]{GGK21}. 

\begin{remark}\label{cGGK}
    Our proof shows that $c_{\mathrm{GGK}}$ can be taken as $\frac{c_{\mathrm{DGH}}}{g+1}$, where $c_{\mathrm{DGH}}$ is as in the proof of Proposition \ref{HeightGap}.
\end{remark}
\begin{proposition}\label{HeightGap}
    Let $g\ge 2$ and $D\ge 1$. There are  constants $c_{\mathrm{DGH}}\coloneqq c_{\mathrm{DGH}}(g,D)>0, ~c'_{\mathrm{DGH}}\coloneqq c'_{\mathrm{DGH}}(g,D)$ with the following property. Let $(A,\phi_L,\nu)\in\cal{A}_{g,3}(\Qbar)$, and let $X\subseteq A$ be an irreducible closed subvariety that generates $A$, with $\deg_L(X)\le D$ and $X^{\circ}\neq\varnothing$. Then 
    $$\hat{h}_{A,L}(X)\ge \frac{c_{\mathrm{DGH}}}{g+1} \max\{1,h_{\bm\omega}(A)\}-\frac{c'_{\mathrm{DGH}}}{g+1}.$$
\end{proposition}

\begin{proof}
By \cite[Proposition 4.1, 4.2]{GGK21} we can find positive constants $c_{\mathrm{DGH}}\coloneqq c_{\mathrm{DGH}}(g,D),~c'_{\mathrm{DGH}}\coloneqq c'_{\mathrm{DGH}}(g,D)$ and $c_3 = c_3(g, D)$ such that for all $g$-dimensional abelian varieties $A$ and subvarieties $X\subset A$ as in the statement,  the set 
$\{P \in X^{\circ}(\ovl{\bb{Q}}) \mid \h_{A,L}(P) \leq c_{\mathrm{DGH}} \max\{h_{\bm\omega}(A),1\}-c'_{\mathrm{DGH}}\}$
is contained in some Zariski closed set $X'\subsetneq X$ with $\deg_L(X')\le c_3$. In particular,  
\begin{align*}
    \inf_{P\in X^{\circ}\setminus X'}\h_{A,L}(P) \ge c_{\mathrm{DGH}}\max\{h_{\bm\omega}(A),1\}-c'_{\mathrm{DGH}}.
\end{align*}
Since $X^{\circ}$ is Zariski open by \cite{Kaw80}, Zhang's inequality \cite[Theorem 1.10]{Zha95b} implies that 
\begin{align*}
\hat{h}_{A,L}(X)\ge c_{\mathrm{DGH}}/ (g+1) \max\{h_{\bm\omega}(A),1\}-c'_{\mathrm{DGH}}/(g+1).
\end{align*}
The result follows. 
\end{proof}

\begin{remark}\label{HeightGapRemark}
In fact, Gao--Ge--K\"uhne's results in their full strength would allow us to choose $c'_{\mathrm{DGH}}=0$. However, we do need this stronger version and in particular K\"uhne's equidistribution result \cite{Kuh21}.
\end{remark}

\noindent\emph{Proof of Theorem \ref{GeometricHeightGap2}:}
Let $K=k(B)$ where $k$ is an algebraically closed field of characteristic zero and $B$ is a regular projective variety over $k$. 
Recall that as in \S\ref{sec: generic}, there is a transcendental extension $k''$ of $k$ and a curve $B''$ over $k''$ such that 
$\h_{A,L,K}(X)= \h_{A',L',k''(B'')}(X')$ and  $h_{\bm\omega,K}(A)=h_{\bm\omega,k''(B'')}(A')$ where $A'$, $L'$, $X'$ are appropriate base extensions of $A,L,X$; see Proposition \ref{prop: generic}. 
We may thus assume that $\dim B=1$. 
It is also clear that we may assume that $k$ is a finitely generated extension of $\Qbar$. 
We will argue by induction the transcendence degree of $k$ over $\Qbar$.

First assume that $k=\Qbar$ for the base case. 
Let $(A,\phi_L,\nu)\in \cal{A}_{g,3}(\Qbar(B))$ and let $X \subseteq A_{\Qbar(B)}$. 
Choose a model $(U,\, \pi : \mathcal{A} \to U,\, \mathcal{L},\, \mathcal{X})$ over $B$. 
Since $X^{\circ} \neq \emptyset$, we also have $\cal{X}_s^{\circ} \not = \emptyset$ by Hindry \cite{Hin88} after possibly shrinking $U$. 
Since $X$ generates $A$, by  \cite[Lemma 2.5]{GGK21}, if we shrink $U$ so that the restriction of $\cal{A}$ to $U$ is a principally polarized abelian scheme, we have that $\cal{A}|_U = (\cal{X}-\cal{X}) + \cdots + (\cal{X}-\cal{X})|_U$ where we add $g$ copies of $\cal{X}-\cal{X}$ fiberwise. 
Thus specializing to any $t \in U$ gives us $\cal{A}_t$ is generated by $\cal{X}_t$ 
and we may apply Proposition \ref{HeightGap} to infer that 
\begin{align}\label{gap on fibers}
\h_{\cal{A}_t, \cal{L}_t}(\cal{X}_t) \geq \frac{c_{\mathrm{DGH}}}{g+1} \max\{h_{\omega}(\cal{A}_s),1\}-\frac{c'_{\mathrm{DGH}}}{g+1},
\end{align}
for all $t\in U$. 
Fix a Weil height $h_B$ on $B$ corresponding to a degree $1$ divisor. 
From \eqref{gap on fibers} and \cite[Main Theorem]{GV24}, we infer that 
\begin{align}
\hat{h}_{A,L,\Qbar(B)}(X)=\lim_{\substack{h_B(t) \to \infty \\ t \in U(\Qbar)}}\frac{\hat{h}_{\cal{A}_t, \cal{L}_t}(\cal{X}_t)}{h_{B}(t)}\ge  \frac{c_{\mathrm{DGH}}}{g+1} \lim_{h_B(t)\to\infty} \frac{h_{\bm\omega}(\cal{A}_t)}{h_B(t)}= \frac{c_{\mathrm{DGH}}}{g+1} h_{\bm\omega,\Qbar(B)}(A). 
\end{align}
The conclusion follows. 

Now suppose that the result holds for all $k$ with $\mathrm{trdeg} k=n\ge 0$. 
Let $k_{n+1}$ be a finitely generated extension of $\Qbar$ with transcendence degree $n+1$, so that $k_{n+1}=\Qbar(V_{n+1})$ for a normal projective variety $V_{n+1}$, irreducible over $\Qbar$ and of dimension $n+1$; see e.g \cite[Lemma 1.4.10]{BG06}.  
Fix $n$ elements $t_1,\ldots,t_n\in \Qbar(V_{n+1})$  that are algebraically independent, and let $k_n$ denote the algebraic closure of $\Qbar(t_1,\ldots,t_n)$ in $\Qbar(V_{n+1})$. 
It is easy to see that $k_n$ is a finitely generated extension of $\Qbar$ and can be written as $k_n=\Qbar(V_n)$ for a regular projective $n$-dimensional variety $V_n$. 
Furthermore, $\Qbar(V_{n+1})$ is a finite extension of $k_n$ of transcendence degree $1$, so that 
$\Qbar(V_{n+1})=k_n(C)$ for a regular, geometrically irreducible, projective curve $C$ over $k_n$.  

Consider now $(A,\phi_L,\nu)\in \cal{A}_{g,3}((k_n(C))(B))$, let $X \subseteq A_{(k_n(C))(B)}$ and choose a model $(U,\, \pi : \mathcal{A} \to U,\, \mathcal{L},\, \mathcal{X})$ over $B$. 
Recall that $B$ is a regular irreducible projective curve over $k_n(C)$ by assumption. 
So for each $t\in U(\overline{k_n(C)})$ the fiber $A_t$ is a principally polarized abelian variety with polarization given by $L_t$, $X_t$ and both are defined over the residue field of $t$ in $U$ which is a finite extension of $k_n(C)$. 
Shrinking $U$ as before we may also ensure that $X_t$ generates $A_t$. 
Applying the induction hypothesis we have 
\begin{align}\label{gap induction}
\h_{\cal{A}_t, \cal{L}_t, \ovl{k_n(C)}}(\cal{X}_t) \geq \frac{c_{\mathrm{DGH}}}{g+1}h_{\omega, \ovl{k_n(C)}}(\cal{A}_t),
\end{align}
for all $t\in U(\overline{k_n(C)})$. 
Choosing a degree one height on $B$ from viewing $B$ as a curve over $k_n(C)$ and applying \eqref{gap induction} to parameters of growing height as in the base case, we infer by our specialization Proposition \ref{specialization_av} that 
\begin{align}
\hat{h}_{A,L,(\overline{k_n(C)})(B)}(X)\ge \frac{c_{\mathrm{DGH}}}{g+1} h_{\bm\omega,(\overline{k_n(C)})(B)}(A). 
\end{align}
The result follows by induction. \qed

\subsection{A height comparison}\label{diff_height}
Here we establish a comparison between the N\'eron--Tate height on an abelian variety and the Weil height associated to an embedding in projective space. 
In the case of number fields such a comparison is done explicitly in \cite{DP02} using the theory of algebraic theta functions by Mumford \cite{Mum66}. 
It should be possible to extend their proofs to the setting of function fields but as we do not need explicit estimates, we instead opt for a softer approach inspired by Silverman's specialization result \cite{Sil83} and Dimitrov--Gao--Habegger \cite{DGH21}. 

\begin{theorem} \label{HeightGeometricCompare1}
Let $g\ge 2$. There exist $N\coloneqq N(g)\in\bN$ and $c_{\mathrm{S}}\coloneqq c_{\mathrm{S}}(g)>0$ such that for every characteristic zero function field $K$, and every $(A,\phi_L,\nu)\in \cal{A}_{g,3}(\ovl{K})$, there is an embedding $\iota_A: A\hookrightarrow \bP^N$ such that $\iota_A^{*}O(1)\iso L$ and such that 
$$\sup_{x\in  A(\ovl{K})}|\hat{h}_{A,L,K}(x) - h_K(\iota_A(x))| \leq c_{\mathrm{S}} h_{\bm\omega,K}(A).$$
\end{theorem}

\begin{proof}
We recall the setup in \cite[Section 6]{DGH21}.
 We let $\frak{A}_g$ be the universal abelian scheme over $\cal{A}_{g,3}$. Then Dimitrov--Gao--Habegger construct an embedding $\iota: \frak{A}_g \xhookrightarrow{} \cal{A}_{g,3} \times \bb{P}^n$ over $\ovl{\bb{Q}}$ \cite[Section 6]{DGH21}. 
 Note also that there is a universal symmetric ample line bundle $\cal{L}$ on $\frak{A}_g$ \cite[Section 6.1]{DGH21} that provides us a fiber-wise Neron--Tate height $\h_{A_s,\cal{L}_s}$ on each fiber of $s \in \cal{A}_{g,3}$. 
The restriction $\cal{L}_s$ to the point $s$ defined by the ppav $(A,L)$ satisfies
$$\h_{A,\cal{L}_s} = 2N \h_{A,L}$$
where $N \geq 4$ is a fixed integer. 
We will take $N = 4$ for our construction. 
Let $K = k(B)$ be the function field of some projective normal variety $B$ over the algebraically closed field $k$. Let $A/K$ be a principally polarized abelian variety, which after passing to a finite extension we may assume comes with a level $3$-structure. 
This defines for us a $K$-point on $\cal{A}_{g,3}$. 
The embedding $\frak{A}_g \xhookrightarrow{} \cal{A}_{g,3} \times \bb{P}^n$ gives us an embedding $\iota: A \xhookrightarrow{} \bb{P}^n_K$, where the restriction of $O(1)$ gives us a symmetric very ample line bundle $L$ on $A$. 
We may then define the Neron--Tate height $\h_{A,L}$ with respect to $L$. Then under $\iota$, by \cite[Theorem A.1]{DGH21}, there exists a constant $c > 0$ depending only on $g$ such that 
\begin{equation} \label{eq: HeightCompare1}
|\h_{A_s,\cal{L}_s}(x) - h(\iota(x))| \leq c_{S} \max\{1, h_{\omega}(A_s)\}
\end{equation}
for all $s \in \cal{A}_{g,3}(\ovl{\bb{Q}})$ and $x \in A_s(\ovl{\bb{Q}})$. We now wish to deduce the same inequality for all $s \in \cal{A}_{g,3}(K)$ and $x \in A_s(\ovl{K})$ where $K$ is a function field over characteristic zero. This follows from a similar argument to Theorem \ref{GeometricHeightGap2} by the transcendental generic curve of \S\ref{sec: generic}, and then inducting on the transcendence degree. We will not repeat the argument here as it is very similar. 
\end{proof}

\subsection{Proof of Theorems \ref{uniform_geometric_Bogomolov}  and \ref{uniform CGHX}}\label{proof of uniform B}

\leavevmode

\medskip

\noindent\emph{Proof of Theorem \ref{uniform_geometric_Bogomolov}:}
Let $A$ be an abelian variety over $K$ with an ample line bundle $L$ and let $X\subset A_{\overline{K}}$ be an irreducible subvariety that generates $A$ with $\deg_L(X)\le D$ and $X^{\circ}\neq \emptyset$. We first handle the case where $A$ is principally polarized by $L$ and has level $3$-structure. In this case, Theorem \ref{GeometricHeightGap2} gives us 
$$\h_{A,L}(X) \geq \frac{c_{\mathrm{DGH}}}{g+1} h_{\omega,K}(A).$$
Since $A$ is non-isotrivial, we have that $h_{\omega,K}(A) > 0$. By Theorem \ref{HeightGeometricCompare1}, we have an embedding to projective space $\iota: A \xhookrightarrow{} \bb{P}^N$ such that 
$$\sup_{x \in A(\ovl{K})} |\h_{A,L,K}(x) - h_K(\iota(x))| \leq c_{\mathrm{S}} h_{\omega}(A).$$
We are thus in the setting to apply Theorem \ref{quantitative Zhang2} which gives us the result in the case of principally polarized abelian varieties with level $3$-structure,

To handle the general case, recall by \cite[Lemma 2.2]{GGK21}, that after replacing $K$ with a finite extension, there exists an abelian variety $A_0/K$, an ample line bundle $L_0$ on $A_0$ which gives a principal polarization and an isogeny $u_0: A \to A_0$ such that $u_0^* L_0 \simeq L$, with $\deg(u_0)=\deg_L(A)/g!$. Note that $\deg(u_0)$ is uniformly bounded in terms of $g$ and $D$, since by \cite[Lemma 2.5]{GGK21} and our assumptions we know that $\deg_L(A)$ is bounded solely in terms of $g$ and $D$. We may further assume that $A_0$ has level $3$-structure. It is easy to see that $(u_0(X))^{\circ}\neq \emptyset$ and by the projection formula we also have 
$$\deg_{L_0}(u_0(X)) \leq \deg_L(X) \le D.$$
Since we have established Theorem \ref{uniform_geometric_Bogomolov} for ppavs with level $3$-structure, we thus obtain positive constants $c' = c'(g,D)$ and $E'\coloneqq E(g,D)\in\bR$ and a strict subvariety $Y' \subset u_0(X)$ with $\deg_{L_0} Y' \leq E'$ such that the set 
$$\{x \in u_0(X)(\ovl{K}) \mid \h_{A_0,L_0}(x) \leq c h_{\Fal}(A_0)\},$$
is contained in $Y'$. 
Since by definition, $h_{\Fal}(A) = h_{\Fal}(A_0)$ and also as $u_0$ is an isogeny, it follows that 
$$\h_{A,L}(x) = (\deg u_0) \h_{A_0,L_0}(u_0(x)),$$
for all $x\in A(\Kbar)$. Setting $Y\coloneqq u_0^{-1}(Z') \cap X$, and noting that $\deg_L(Y) \leq \deg_L(u_0^{-1}(Y')) \cdot  \deg_L(X)$ which is bounded purely in terms of $g$ and $D$, the result follows. 
\qed

\bigskip

\noindent\emph{Proof of Theorem \ref{uniform CGHX}:}
We proceed by induction on $\dim V$. Let $B$ be the abelian subvariety that $V$ generates. 
Then $B$ is defined over a finite extension of $K$ that is bounded in terms of $g$. 
Since we have also assumed that $\Tr_{\ovl{K}/\bb{C}}(A) = 0$ and that $M$ is very ample, we infer that there exists $\delta = \delta(g,K)$ such that 
$$h_{\Fal}(B) \geq \delta;$$
see for example \cite[Lemma 4.9]{Yua24}. By Theorem \ref{uniform_geometric_Bogomolov}, we know there exists $\eps' = \eps'(g,D)$ and $D' = D'(g,D)$ such that for any $W$ generating $B$ which is not equal to $B$, all points of $W^{\circ}$ with canonical height at most $(\eps')\delta$ are contained in a subvariety of $W$ with degree at most $D'$.  
Let $\eps = \frac{\eps'}{4}$. If $V^{\circ}$ has no point of canonical height at most $\eps$, then we are done. 
Otherwise, let $p\in V^{\circ}$ be a point of height at most $\eps$. 
Then $W:=V-p \subseteq B$. 
If $W= B$, then $V^{\circ} = \emptyset$ and we are done. If $W \not = B$, then all points of height at most $\eps'$ are contained in a union of irreducible subvarieties $\{W_i\}$, whose total degree is at most $D'$. 
Note that $W^{\circ} + p = V^{\circ}$. Hence the points of height $< \eps'$ in $V^{\circ}$ are necessarily contained in $p+W_i$. We are then done by induction as $V^{\circ} \cap (p+ W_i) \subseteq (p + W_i)^{\circ}$. If $W = B$ then $V = p+B$ and $V^{\circ} = \emptyset$ and we are done too. 
\qed

\bibliographystyle{hplain}

\end{document}